\documentclass[12 pt]{amsart}
\usepackage{amssymb, amsmath, amsthm, amsfonts, amscd}
\usepackage{xcolor}
\usepackage{datetime}
\usepackage{enumerate}
\usepackage{mathabx}
\usepackage{blkarray}

\newtheorem{theorem}{Theorem}[section]
\newtheorem{proposition}[theorem]{Proposition}
\newtheorem{lemma}[theorem]{Lemma}
\newtheorem{corollary}[theorem]{Corollary}

\theoremstyle{definition}
\newtheorem{definition}[theorem]{Definition}
\newtheorem{example}[theorem]{Example}

\newcommand{\clm}{\mathcal{M}}
\newcommand{\cln}{\mathcal{N}}
\newcommand{\clh}{\mathcal{H}}
\newcommand{\cld}{\mathcal{D}}

\newcommand{\clb}{\mathcal{B}}

\newcommand{\clu}{\mathcal{U}}
\newcommand{\cly}{\mathcal{Y}}

\newcommand{\norm}[1]{\left\lVert#1\right\rVert}

\theoremstyle{remark}
\newtheorem{remark}[theorem]{Remark}

\numberwithin{equation}{section}

\title{Characteristic functions and Colligations}

\author{Neeru Bala, Santanu Dey, Kalpesh J. Haria and M. N. Reshmi }

\begin{document}
	\begin{abstract}
		The characteristic function of row contractions and characteristic function of liftings of row contractions are multi-analytic operators which are complete invariants up to unitary equivalence for row contractions and liftings of row contractions, respectively. We provide alternate proofs for these properties of characteristic functions using colligations. Co-isometric observable colligations with certain class of basic operators are characterized.  Blaschke factor based transformations of the characteristic function of lifting are studied.
	\end{abstract}
	\maketitle

	\vspace{0.5cm}
	\noindent {\bf keywords:} Row contractions, contractive lifting, Fock space, characteristic functions, multi-analytic operators, completely non-co-isometric, minimal contractive lifting, colligations

	\vspace{0.5cm}

	\section{Introduction}
	One of the well studied classes of operators on Hilbert spaces is the class of normal operators and the spectral theorem for normal operators plays a vital role in its prominence. For a bounded linear operator $T$ on a Hilbert space, the resolvent function is $(T-\lambda I)^{-1}$ (for complex $\lambda$ where the operator $(T-\lambda I)^{-1}$ is defined). The spectrum of an operator and the spectral theorem  is directly related to the resolvent function, which is an operator valued analytic function.  The resolvent function is very important in the study of normal operators, but it provides relatively less information about the structure of non-normal operators. A bounded operator $T$ on a Hilbert space is called a {\it contraction} if $\|T\| \leq 1.$ Nagy and Foias (cf. \cite{NAGY}) associated an operator valued analytic function to a contraction $T$ on a Hilbert space which is called the characteristic function of $T$. These functions  gives valuable information about the structure of contractions.  Nagy-Foias characteristic function is useful in studying non-normal operators. The book by Nagy and Foias \cite{NAGY} is a good reference for the results related to non-normal contractions and the dilation theory. Further, Popescu \cite{POP1,POP2} explored the characteristic function in the non-commuting setup for row contractions. 
	
	Let $\clh_T$ be a Hilbert space such that $\{T_i\}_{i=1}^d\subseteq\mathcal{B}(\clh_T)$. Then $\underline{T}=(T_1,T_2,\ldots, T_d)$ is called a {\it row contraction}, if $\underset{i=1}{\overset{d}{\sum}}T_iT_i^*\leq I$, which is equivalent to saying that the matrix $\begin{bmatrix}
		T_1&T_2&\ldots&T_d
	\end{bmatrix}$ is a contraction. The characteristic function $M_{T}:\Gamma\otimes\mathcal{D}_T\rightarrow\Gamma\otimes\mathcal{D}_{*,T}$ of the row contraction $\underline{T}$ is defined by 
	\begin{equation}\label{pop charac fun}
		{	M_{T}}=-I_{\Gamma}\otimes \underline{T}+(I_{\Gamma}\otimes D_{*,T})({I_{\Gamma\otimes \clh_T}}-(\underline{R}\otimes I_{\clh_T})(I_{\Gamma} \otimes \underline{T}^*))^{-1}(\underline{R} \otimes I_{\clh_T})(I_{\Gamma} \otimes D_T),	
	\end{equation}
	where $\mathcal{D}_T$, $\mathcal{D}_{*,T}$ are the defect spaces associated to $\underline{T}$ and $\Gamma$ is the full Fock space defined on $\mathbb{C}^d$. Further, the symbol of $M_T$ is defined as $\theta_T=M_T|_{e_0\otimes\mathcal{D}_T}$. The articles
	\cite{POP1,POP2,POP3} are a few references on the fundamentals of the dilation theory for row contractions.
	
	For a row contraction $\underline{C}=(C_1,C_2,\ldots,C_d)$ on a Hilbert space $\clh_C$, we say a row contraction $\underline{E}=(E_1,E_2,\ldots,E_d)$ is a  {\it contractive lifting} of $\underline{C}$ by $\underline{A}$ on a Hilbert space $\clh_E=\clh_C\oplus \clh_A$, if we have
	\[E_i=\begin{bmatrix}
		C_i&0\\
		B_i&A_i
	\end{bmatrix}\text{ for }i=1,2,\ldots,d\] 
	with respect to this decomposition of $\clh_E$ where $\clh_A$ is some Hilbert space, and $B_i$ and $A_i$ are some bounded operators. Contractive liftings have been investigated in the context of commutant lifting theorem (\cite{NAGY}), weak Markov processes (\cite{GOHM}), Scattering theory (\cite{DEY gen}), etc. The book by Foias and Frazho \cite{FRAZHO} is an excellent source for topics that demonstrate the theoretical and practical aspects of contractive liftings. 
	\begin{definition}
		Let $\underline{C}$ be a row contraction  on a Hilbert space $\clh_C.$ A contractive lifting $\underline{E}$ of $\underline{C}$ on a Hilbert space $\clh_E \supset \clh_C$ is called {\it minimal,} if 
		\begin{align*}
			\clh_{E} = \overline{span}\{E_{\alpha}x: x\in \clh_C\;\;\text{for all }\,\alpha\in\tilde{\Lambda}\}.
		\end{align*}
	\end{definition}
	S. Dey and R. Gohm \cite{DEY} associated a characteristic function $M_{C,E}:\Gamma\otimes\mathcal{D}_E\rightarrow\Gamma\otimes\mathcal{D}_C$ with symbol $\theta_{C,E}:\mathcal{D}_E\rightarrow\Gamma\otimes\mathcal{D}_C$ to a minimal contractive lifting. We have the following expression for $\theta_{C,E}$:
	For $h_c\in \clh_C$,
	\begin{align}\label{chara fun deygohm1}
		\theta_{C,E}(e_0 \otimes D_E \iota^{\clh_E}_j(h_c))=e_0\otimes(D_Ch_C-\gamma D_{*,A}B_j h_c)-\underset{|\alpha|\geq 1}{\sum}e_{\alpha}\otimes \gamma D_{*,A}A_{\alpha}^*B_j h_c,\end{align}
	and for $h_a\in \clh_A$,
	\begin{align}\label{chara fun deygohm}
		\theta_{C,E}(e_0\otimes D_E \iota^{\clh_E}_j(h_a))=-e_0\otimes\gamma\underline{A}D_A \iota^{\clh_A}_j(h_a) +\underset{j=1}{\overset{d}{\sum}}e_j\otimes\underset{\alpha}{\sum}e_{\alpha}\otimes\gamma D_{*,A}A_{\alpha}^*P_jD_A^2 \iota^{\clh_A}_j(h_a),
	\end{align}
	where the function  $\iota^\clh_j: \clh \to \underset{i=1}{\overset{d}{\oplus}} \clh$ for any Hilbert space $\clh$ and $j=1,\ldots,d$ is defined by $\iota^\clh_j(h)=(0, \ldots, 0, h, 0 , \ldots, 0)$ where $h$ is in the $j^{th}$ component.
	
	S. Dey and R. Gohm \cite{DEY} studied properties of characteristic function in various settings. Also, different subclasses of contractive liftings are studied in \cite{DEY,DEYII,DEY minimal}.

	The paper is organised as follows: In section $2$, we fix some notations and recall the results that are used in subsequent sections. In section $3$, we show that, if Popescu characteristic function (\ref{pop charac fun}) coincides for two different completely non-co-isometric row contractions, then the row contractions are unitarily equivalent. This is already proved by Popescu \cite{POP2}, but here we give a different proof using colligation matrices. In section $4$, we give the following compact form for $\theta_{C,E}$, that is 
	\begin{align} \label{comp}
		\theta_{C,E}=\begin{bmatrix}
			D_{*,\gamma}&(I_{\Gamma}\otimes\gamma)\theta_A
		\end{bmatrix}
		{\sigma},
	\end{align}
	where $\sigma:\mathcal{D}_E\rightarrow\mathcal{D}_{*,\gamma}\oplus\mathcal{D}_A$ is a unitary operator. Also, we prove that $\theta_{C,E}$ can be realized as a transfer function for the following colligation matrix
	\[V=\begin{bmatrix}
		\underline{A}^*& D_AP_{D_A}\sigma\\
		\gamma D_{*,A}& (D_{*,\gamma}P_{D_{*,\gamma}}\sigma-\gamma \underline{A}P_{D_A}\sigma)
	\end{bmatrix}: \clh_A \oplus \mathcal{D}_E \to \clh^d_A \oplus \mathcal{D}_C.\]
	
	For a contraction  $T$ on  a Hilbert space and any $a$ in the open unit disc $\mathbb{D}$, define the operator $T_a := (T-aI)(I-\overline{a}T)^{ -1}.$ In section $5$, we prove that if $E$ is a minimal contractive lifting of a row contraction $C,$ then $E_a$ is a minimal contractive lifting of $C_a$, and the symbol $\theta_{C_a,E_a}(\lambda)$ of the characteristic function of the lifting $E_a$ coincides with $\theta_{C,E}(\frac{\lambda+a}{1+\overline{a}\lambda}).$ 

In section $6$, we show that if $\underline{E}$ and $\underline{E}'$ are two minimal contractive liftings of a row contraction $\underline{C}$, then $\underline{E}$ and $\underline{E}'$ are unitarily equivalent if and only if $\theta_{C,E}$ and $\theta_{C,E'}$ are equivalent. If $M_{\theta}: \Gamma \otimes \cld \to \Gamma \otimes \cld_C$ be a contractive multi-analytic operator with an injective symbol $\theta$, then $\theta$ has a decomposition as given in Equation (\ref{comp}) with $\gamma$ resolving (refer Equation (\ref{eqn gamma})) and $M_{\theta_A}$ is purely contractive and satisfy the Szeg\"o condition for some row contraction $\underline{A}$.  The converse of this statement also holds. We establish that for two contractive multi-analytic operators  $M_{\theta}: \Gamma \otimes \cld \to \Gamma \otimes \cld_C$ and $M_{\hat{\theta}}: \Gamma \otimes \clm \to \Gamma \otimes \cld_C$  with injective symbols, the symbols $\theta$ and $\hat{\theta}$ are equivalent if and only if the ordered pairs $(\gamma, \theta_A)$ and $(\gamma, \theta_{\hat{A}})$ corresponding to the liftings of $\underline{C}$ associated to $M_{\theta}$ and $M_{\hat{\theta}},$ respectively, are equivalent.

In section $7$, we illustrate our results with some concrete examples. In section $8$, we explore some applications of our result in the theory of colligation matrix and give a characterization of class of co-isometric observable colligations with basic operator $\underline{A}^*$ for which the dimension of the input space equal to the dimension of $\mathcal{D}_A,$ where $\underline{A}$ be a row contraction with a finite dimensional defect space $\mathcal{D}_A.$ 
	
	
	
	\section{ Preliminaries}

	The full Fock space on $\mathbb{C}^d$ is defined by
	\begin{align*}
		\Gamma(\mathbb{C}^d):=\mathbb{C}\oplus\mathbb{C}^d\oplus(\mathbb{C}^d)^{\otimes2}\oplus\cdots\oplus(\mathbb{C}^d)^{\otimes n}\oplus\cdots,
	\end{align*}
	where $e_0=1\oplus0\oplus0\oplus\cdots$ is called the {\it vaccum vector}. To simplify the notation, we use $\Gamma$, for $\Gamma(\mathbb{C}^d)$. We denote the standard ordered basis of $\mathbb{C}^d$ by $\{e_1,e_2,\ldots,e_d \}$. For $x\in\Gamma(\mathbb{C}^d)$, the operators
	$$L_ix=e_i\otimes x \text{ and }R_ix=x\otimes e_i,\,i=1,2,\ldots,d,$$
	are called the {\it left creation} and {\it right creation} operators, respectively. The tuples $(L_1, \ldots, L_d)$ and $(R_1, \ldots, R_d)$ are row contractions. We use the notation $\Lambda$ for the set $\{1,2,\ldots,d\}$ and $\tilde{\Lambda}:=\underset{n=0}{\overset{\infty}{\bigcup}}\Lambda^n$, where $\Lambda^0=\{0\}$. For $\alpha=(\alpha_1,\ldots,\alpha_n)\in\Lambda^n$, we define $|\alpha|:=\alpha_1+\alpha_2+\cdots +\alpha_n$ and $e_\alpha= e_{\alpha_1} \otimes e_{\alpha_{2}} \otimes \ldots \otimes e_{\alpha_n}$. 
	
	For a row contraction $\underline{T}=(T_1,T_2,\ldots, T_d)$ on a Hilbert space $\clh_T$, the {\it defect operators} and {\it defect spaces} associated with $\underline{T}$ are defined by
	\begin{align*}
		D_T&=(I-\underline{T}^*\underline{T})^{1/2}:\underset{i=1}{\overset{d}{\oplus}}\clh_T\rightarrow\underset{i=1}{\overset{d}{\oplus}}\clh_T,\\
		D_{*,T}&=(I-\underline{T} \, \underline{T}^*)^{1/2}:\clh_T\rightarrow \clh_T,\text{ and}\\
		\mathcal{D}_T&=\overline{\text{ran}}\,D_T,\,\mathcal{D}_{*,T}=\overline{\text{ran}}\,D_{*,T},
	\end{align*}
	respectively. For $\alpha=(\alpha_1,\ldots,\alpha_n)\in\Lambda^n$, $T_{\alpha}=T_{\alpha_1}\ldots T_{\alpha_n}$. We say the row contraction $\underline{T}$ is {\it completely non-co-isometric} (or c.n.c. in short), if
	$$\clh_T^{1}=\left\{h\in \clh_T:\underset{|\alpha|=n}{\sum}\|T_{\alpha}^*h\|^2=\|h\|^2 \text{ for all }n\in\mathbb{N}\right\}=\{0\}.$$
	It is easy to observe that $D_{*,T}$ is an injective operator, if  $\underline{T}$ is a c.n.c. row contraction.
	
	Let $\underline{T}=(T_1,T_2,\ldots,T_d)$ be a c.n.c. row contraction on a Hilbert space $\clh_T$. The {\it characteristic function} {$M_{T}:\Gamma\otimes \cld_T\rightarrow\Gamma\otimes \cld_{*,T}$  of the row contraction $\underline{T}$ is defined by Equation \eqref{pop charac fun} 
			and its symbol $\theta_{T}:\mathcal{D}_T\rightarrow\Gamma\otimes\mathcal{D}_{*,T}$ is $\theta_{T}=M_T|_{e_0\otimes\mathcal{D}_T}.$} Then
		\begin{equation}\label{popescu chara fn}
			\theta_{T}=-e_0 \otimes \underline{T}+(I_{\Gamma}\otimes D_{*,T})({I_{\Gamma\otimes \clh_T}}-(\underline{R}\otimes I_{\clh_T})(I_{\Gamma} \otimes \underline{T}^*))^{-1}(\underline{R} \otimes I_{\clh_T})(I_{\Gamma} \otimes  D_T).
		\end{equation} 
		{We can identify $\mathcal{D}_T$ with closed linear span of elements of the form $e_0\otimes D_Th$ for $h\in H$. Then $\theta_{T}$ can be written as 
			\begin{equation*}
				\theta_{T}=- \underline{T}+ I_{\Gamma}\otimes D_{*,T}(I_{\Gamma\otimes \clh_T}-(\underline{R}\otimes I_{\clh_T})(I_{\Gamma} \otimes\underline{T}^*))^{-1}(\underline{R} \otimes I_{\clh_T})(I_{\Gamma} \otimes  D_T).
			\end{equation*}
			
			A contractive colligation on $d$-variable is an operator 
			\[W=\begin{bmatrix}
				T&F\\
				G&H
			\end{bmatrix}=\begin{bmatrix}
				T_1&F_1\\
				T_2&F_2\\
				\vdots&\vdots\\
				T_d&F_d\\
				G&H
			\end{bmatrix}:\clh_1\oplus \clh_2\rightarrow \clh_1^d\oplus \clh_3,\] 
			where $\clh_1, \clh_2$ and $\clh_3$ are Hilbert spaces,  and $\clh^d_1:=\underset{i=1}{\overset{d}{\oplus}} \clh_1.$ The spaces $\clh_1, \clh_2$ and $\clh_3$ are called {\it the state space, the input space} and {\it the output space,} respectively. The operator 
			$T$ is called the {\it basic operator} of the colligation and the function 
			\begin{equation} \label{transfer}
				\Theta_W=H+G({I_{\Gamma\otimes \clh_1}}-(\underline{R}\otimes I_{\clh_1})(I_{\Gamma} \otimes  T))^{-1}(\underline{R}\otimes I_{\clh_1})(I_{\Gamma}\otimes F)
			\end{equation}
			is called the associated {\it transfer function} of the colligation.  Note that $\Theta_W\in\Gamma\otimes\mathcal{B}(\clh_2,\clh_3)$. We say $W$ is {\it observable}, if the observability operator $\mathcal{O}_{G,T}:\clh_1\rightarrow \Gamma\otimes \clh_3$ defined by
			\begin{align}
				\mathcal{O}_{G,T}=G({I_{\Gamma\otimes \clh_1}}-(\underline{R}\otimes I_{\clh_1})(I_{\Gamma} \otimes  T))^{-1}
			\end{align}  
			is an injective operator.
			
			{For two Hilbert spaces $\clh_1$ and $\clh_2$, let $\mathcal{M}_{nc,d}(\clh_1,\clh_2)$ consists of functions $\Theta\in\Gamma\otimes\mathcal{B}(\clh_1,\clh_2)$ for which $M_{\Theta}f=\Theta f$ is a well defined bounded linear operator from $\Gamma\otimes \clh_1$ to $\Gamma\otimes \clh_2$. From \cite{BALL}, we know that $M_{\Theta}(L_i\otimes I_{\clh_1})=(L_i\otimes I_{\clh_2})M_{\Theta}$ for $i=1,2,\ldots,d$. We define
				\begin{align*}
					\mathcal{S}_{nc,d}(\clh_1,\clh_2)=\{\Theta\in\mathcal{M}_{nc,d}(\clh_1,\clh_2):\|M_{\Theta}\|\leq 1\}.
				\end{align*}
				For a contractive colligation matrix $W,$ we know that $\Theta_W\in\mathcal{S}_{nc,d}$, by \cite[Theorem 1.1]{BALL}.}
			
			We use the notion of unitary equivalence of two colligation matrices defined as follows:
			\begin{definition}\cite[Page 527-528]{BALL}
				Let \[W=\begin{bmatrix}
					T&F\\
					G&H
				\end{bmatrix}:\clh_1\oplus \clh_2\rightarrow \clh_1^d\oplus \clh_3\] and \[W'=\begin{bmatrix}
					T'&F'\\
					G'&H'
				\end{bmatrix}:\clh_1'\oplus \clh_2\rightarrow \clh_1'^d\oplus \clh_3\] be two colligations where $\clh_1, \clh_2, \clh_3$ and $\clh_1'$ are Hilbert spaces . These colligations are said to be {\it unitarily equivalent} if there exist a unitary operator $U:\clh_1\rightarrow \clh_1'$ such that
				\begin{align*}
					\begin{bmatrix}
						\underset{i=1}{\overset{d}{\oplus}}U&0\\
						0&I_{\underset{i=1}{\overset{d}{\oplus}}\clh_3}
					\end{bmatrix}\begin{bmatrix}
						T&F\\
						G&H
					\end{bmatrix}=\begin{bmatrix}
						T'&F'\\
						G'&H'
					\end{bmatrix}\begin{bmatrix}
						U&0\\
						0&I_{\clh_2}
					\end{bmatrix}.
				\end{align*}
			\end{definition}
			The following result gives the invariance of the transfer function under unitarily equivalent colligation matrices.
			\begin{lemma}\cite[Corollary 3.9]{BALL}\label{lemma ball}
				Any two observable, co-isometric colligation matrices $W$ and $W'$ for same associated transfer function are unitarily equivalent.
			\end{lemma}

			\section{Characteristic functions for c.n.c. row contractions}
			In this section, our aim is to show that two c.n.c. row contractions with same characteristic function (in the sense of Popescu \cite{POP2}) are unitarily equivalent.
			
			Let $\clh_1,\tilde{\clh}_1,\clh_2$ and $\tilde{\clh}_2$ be Hilbert spaces. We say that two multi-analytic operators {$M_1:\Gamma\otimes \clh_1\rightarrow\Gamma\otimes \tilde{\clh}_1$ and $M_2:\Gamma\otimes \clh_2\rightarrow\Gamma\otimes \tilde{\clh}_2$ {\it coincide,} if there exist unitary operators $U: \clh_1\rightarrow  \clh_2$ and $\tilde{U}: \tilde{\clh}_1\rightarrow\tilde{\clh}_2$ such that
				\begin{equation} \label{coincide}
					(I_{\Gamma}\otimes\tilde{U})M_1=M_2 (I_{\Gamma}\otimes U).
				\end{equation}
				Let $\theta_1$ and $\theta_2$ be the symbols of $M_1$ and $M_2$ respectively. 
				The Equation (\ref{coincide}) is equivalent to 
				$$(I_{\Gamma}\otimes\tilde{U})\theta_1=\theta_2U.$$}
			In this case we also say that $\theta_1$ and $\theta_2$ {\it coincides.}

				\begin{definition}  
					For an operator $\theta : \cld \to \Gamma \otimes \mathcal{L}$ in $\mathcal{S}_{nc,d},$ the $M_\theta$ is said to be {\it purely contractive} if $\| P_{e_0 \otimes \mathcal{L}} \theta (d)\| < \|d\|$ for all $0 \neq d \in \cld.$
				\end{definition}      
				
				\begin{definition}  
					For an operator $\theta : \cld \to \Gamma \otimes \mathcal{L}$ in $\mathcal{S}_{nc,d},$ the $M_\theta$   is said to {\it satisfy Szeg\"o condition} if 
					\[ \overline{\Delta(\Gamma \otimes \cld)} = \overline{\Delta((\Gamma \otimes \cld) \ominus (e_0 \otimes \cld))},\] 
					where $\Delta=(I- M^*_\theta M_\theta )^{\frac{1}{2}}.$
				\end{definition}

				Let $\underline{A}$ be a c.n.c. row contraction. Then $M_{\theta_A}$ is both purely contractive and satisfy Szeg\"o condition.
				It is easy to see that $\theta_{A}$ is also the transfer function for the following colligation matrix}
			\begin{align}\label{eqn colligation popescu}
				W_A=\begin{bmatrix}
					\underline{A}^*& D_A\\
					D_{*,A}&-\underline{A}
				\end{bmatrix}.
			\end{align}
			Also $\theta_{A}\in\mathcal{S}_{nc,d}$ for a c.n.c. row contraction $\underline{A}$.
			\begin{remark}\label{remark 3.1}
				Note that $W_A$ is co-isometric (that is $W_AW_A^*=I$) for any row contraction $\underline{A}=(A_1,A_2,\ldots, A_d)$.
			\end{remark}

			We begin with a simple observation about the colligation matrix $W_A$ defined in Equation (\ref{eqn colligation popescu}). The proof of the following lemma is immediate (cf \cite{POPII}):
			\begin{lemma}\label{lemma sec 3}
				Let $\underline{A}=(A_1,A_2,\ldots,A_d)$ be a c.n.c. row contraction on a Hilbert space $\clh_A$. 
				Then $W_A$ is observable.
			\end{lemma}

			If $\underline{A}$ and $\underline{A}'$ be two unitarily equivalent c.n.c. row contractions, then using direct computation, we know that $\theta_{A}$ and  $\theta_{A'}$ coincides (cf. \cite{POP2}). For the converse, Popescu used the functional model. Here we present an alternate proof for the converse using colligation matrix.
			\begin{theorem}\label{thm Popescu characteristic fn}
				Let $\underline{A}=(A_1,A_2,\ldots A_d)$ and $\underline{A}'=(A_1',A_2',\ldots,A_d')$ be two c.n.c. row contractions on Hilbert spaces $\clh_A$ and $\clh_{A'}$. If the characteristic functions $\theta_{A}$ and $\theta_{A'}$  of row contractions coincide, then $\underline{A}$ and $\underline{A}'$ are unitarily equivalent.
			\end{theorem}
			\begin{proof}
				Let $\theta_{A}$ and $\theta_{A'}$ coincide. Then there exist two unitaries $U_1:\mathcal{D}_A\rightarrow \mathcal{D}_{A'}$ and $U_2:\mathcal{D}_{*,A}\rightarrow \mathcal{D}_{*,A'}$ such that $ (I_{\Gamma}\otimes U_2)\theta_{A}=\theta_{A'}  U_1$.
				
				From Equation (\ref{eqn colligation popescu}), we know that $ (I_{\Gamma}\otimes U_2)\theta_{A}$ and $\theta_{A'} U_1$ are the transfer functions of the  colligation matrices
				{\begin{align*}
						W_A=\begin{bmatrix}
							\underline{A}^*& D_A\\
							U_2D_{*,A}&- U_2\underline{A}	\end{bmatrix}\text{ and }
						W_{A'}=\begin{bmatrix}
							{\underline{A}'}^*& D_{{\underline{A}'}}U_1\\
							D_{*,{\underline{A}'}}&-{\underline{A}'}U_1
						\end{bmatrix},
				\end{align*}}
				respectively. Using Lemma \ref{lemma ball}, Remark \ref{remark 3.1} and Lemma \ref{lemma sec 3}, we deduce that $W_A$ and $W_{A'}$ are unitarily equivalent. Thus there exists unitary $S:\clh_A\rightarrow \clh_{A'}$ such that
				\begin{enumerate}
					\item $(\underset{i=1}{\overset{d}{\oplus}} S) \underline{A}^*= \underline{A}'^*S$,
					\item {$ U_2D_{*,A}= D_{*,{A'}}S$,
						\item $\left(\underset{i=1}{\overset{d}{\oplus}}S\right) D_A= D_{{A'}}U_1$,
						\item $ U_2 \underline{A} = \underline{A}' U_1$.}
				\end{enumerate}
				From above equations, we get that $U_2= S|_{\mathcal{D}_{*,A}}$, $U_1=\left(\underset{i=1}{\overset{d}{\oplus}}S\right)|_{\mathcal{D}_A}$. Hence $\underline{A}$ and $\underline{A}'$ are unitarily equivalent.
			\end{proof}
			\begin{remark}\label{Remark popescu charac fn}
				From the proof of Theorem \ref{thm Popescu characteristic fn}, we observe that, if $\left(\underset{i=1}{\overset{d}{\oplus}}U\right) \underline{A}^* =\underline{A}'^* U$ for some unitary operator $U:\clh_A\rightarrow \clh_{A'}$, then $ \left(I_{\Gamma}\otimes\tilde{U}_1\right)\theta_{A}=\theta_{A'}\left(\underset{i=1}{\overset{d}{\oplus}}\tilde{U}_2\right)$, where  $\tilde{U}_1=U|_{\mathcal{D}_{*,A}}$  and {$\underset{i=1}{\overset{d}{\oplus}}\tilde{U}_2=\left(\underset{i=1}{\overset{d}{\oplus}}U\right)|_{\mathcal{D}_A}.$}
			\end{remark}

			\section{Liftings of row contractions} \label{lrc}
			In this section, we give a compact form of the characteristic function of contractive liftings, which is defined by S. Dey and R. Gohm \cite{DEY}. 
			
			First, we recall the definition of lifting of a row contraction, which we have stated in the introduction. Let $\underline{C}=(C_1,C_2,\ldots,C_d)$ be a row contraction on $\clh_C$. Then tuple  $\underline{E}=(E_1,E_2 \dots, E_d)$ of bounded operators on $\clh_E\supseteq \clh_A$ is called a lifting of $\underline{C}$ if $P_{\clh_C}E_i=C_iP_{\clh_C}$ for $i=1,\ldots,d$ or equivalently, $E_i$ has a matrix representation of the form 
			\[E_i=\begin{bmatrix}
				C_i&0\\
				B_i&A_i
			\end{bmatrix}\text{ for }i=1,2,\ldots,d.\]
			If $\underline{E}$ is a row contraction, then $\underline{E}$ is called a contractive lifting of $\underline{C}$. The following result from \cite{DEY} gives a representation for a contractive lifting of a row contraction.
			\begin{proposition}\cite[Proposition 3.1]{DEY}\label{Prop lifting}
				$\underline{E}=(E_1,E_2,\ldots,E_d)$ on	$\clh_E=\clh_C\oplus \clh_A$ with block matrices
				\[E_i=\begin{bmatrix}
					C_i&0\\
					B_i&A_i,
				\end{bmatrix}\]  for $i=1,2,\ldots,d$, is a row contraction if and only if $\underline{C}$ and $\underline{A}$ are row contractions and there exists a contraction $\gamma:\mathcal{D}_{*,A}\rightarrow \mathcal{D}_C$ such that $\underline{B}^*=D_C\gamma D_{*,A}.$
			\end{proposition} 
			Following the terminology of \cite{DEY}, we say that the contraction $\gamma:\mathcal{D}_{*,A}\rightarrow \mathcal{D}_C$ is {\it resolving}, if for every $h\in \clh_A$, 
			\begin{equation}\label{eqn gamma}
				\gamma D_{*,A}A_{\alpha}^*h=0\,\,\,\,\forall\,\alpha\in\tilde{\Lambda}\text{ implies }D_{*,A}A_{\alpha}^*h=0\,\,\,\,\forall\,\alpha\in\tilde{\Lambda}.
			\end{equation}
			\begin{definition} \label{reduced}
				Let $\underline{C}$ be a row contraction. The contractive lifting $\underline{E}$ of $\underline{C}$  by $\underline{A}$ is called a {\it reduced lifting} if $\underline{A}$ is a c.n.c. row contraction and the function $\gamma$ in Proposition \ref{Prop lifting} is resolving.
			\end{definition}

			Let $\clh_1$ and $\clh_2$ be Hilbert spaces.  
			A bounded operator $M: \Gamma \otimes \clh_1 \to \Gamma \otimes \clh_2$ is said to be {\it multi-analytic} if 
			$$M (L_i \otimes I) = (L_i \otimes I) M,\text{ for }i=1, \ldots, d.$$ 
			Such an operator is determined by its {\it symbol} $\theta:= M|_{e_0 \otimes \clh_1}.$ 
			
			By the remark following \cite[Proposition 3.8]{DEY minimal}, it is important to note that the notions of reduced liftings and minimal contractive  liftings are same.  
			For a minimal contractive  lifting $\underline{E}$ of $\underline{C}$ by $\underline{A}$, 
			a characteristic function $M_{C,E}:\Gamma\otimes\mathcal{D}_E\rightarrow\Gamma\otimes\mathcal{D}_C$ with symbol $\theta_{C,E}:\mathcal{D}_E\rightarrow\Gamma\otimes\mathcal{D}_C$ is defined in Equations \eqref{chara fun deygohm1} and \eqref{chara fun deygohm}. 
		Let $\iota_j^{\clh}$ be the inclusion map from a Hilbert space $\clh$ to $\oplus_{i=1}^d \clh$ as the $j^{th}$ component.   Note that $\theta_{C,E}(e_0\otimes D_E \iota^{\clh_E}_j(h_a))=(I_{\Gamma}\otimes\gamma)\theta_{A}(e_0\otimes D_A \iota^{\clh_A}_j(h_a))$ for $h_a\in \clh_A$, where $\theta_{A} : \mathcal{D}_A \rightarrow\Gamma\otimes\mathcal{D}_{*,A}$ is the characteristic function of $\underline{A}$ and is given by
			\begin{equation} \label{POPchar}
				\theta_{A} (e_0\otimes D_A \iota^{\clh_A}_j(h_a))=-e_0\otimes \underline{A}D_A \iota^{\clh_A}_j(h_a) +\underset{j=1}{\overset{d}{\sum}}e_j\otimes\underset{\alpha}{\sum}e_{\alpha}\otimes  D_{*,A}A_{\alpha}^*P_jD_A^2 \iota^{\clh_A}_j(h_a),
			\end{equation}
			as in page 56 of \cite{POP2}. It was establised in \cite{POPconst} that the expression of $\theta_A$  in Equation (\ref{popescu chara fn}) is the Fourier expansion of Equation (\ref{POPchar}).

		First, we develop another expression for $\theta_{C,E}$ which would be useful  later in this article.
		\begin{lemma}\label{characteristic function lemma}
			Let $\underline{C}=(C_1,C_2,\ldots, C_d)$ be a row contraction on a Hilbert space $\clh_C$ and $\underline{E}=(E_1,E_2,\ldots, E_d)$ be a minimal contractive  lifting of $\underline{C}$ on a Hilbert space $\clh_E=\clh_C\oplus \clh_A$, where
			\[E_i=\begin{bmatrix}
				C_i&0\\
				B_i&A_i
			\end{bmatrix}\text{ for }1\leq i\leq d.\]
			Let $\gamma:\mathcal{D}_{*,A}\rightarrow \mathcal{D}_C$ be a contraction such that $\underline{B}^*=D_C\gamma D_{*,A}.$ Then 
			\begin{align*}
				\theta_{C,E}=\begin{bmatrix}
					D_{*,\gamma}&(I_{\Gamma}\otimes\gamma)\theta_A
				\end{bmatrix}
				{\sigma},
			\end{align*}
			where $\sigma:\mathcal{D}_E\rightarrow\mathcal{D}_{*,\gamma}\oplus\mathcal{D}_A$ is a unitary operator defined by 
			\begin{align}\label{eqn for sigma}
				\sigma D_E=\begin{bmatrix}
					D_{*,\gamma}D_C&0\\
					-\underline{A}^*\gamma^*D_C&D_A
				\end{bmatrix}.
			\end{align}
		\end{lemma}	
		\begin{proof} In \cite{BDR}, it is shown that the map $\sigma$ is a unitary.
			For $h_a\in \clh_A$ and $j=1, \ldots, d$, we have
			\begin{align*}
				\begin{bmatrix}
					D_{*,\gamma}&(I_{\Gamma}\otimes\gamma)\theta_A
				\end{bmatrix}{\sigma}  D_E \iota^{\clh_E}_j(h_a)
				=&\begin{bmatrix}
					D_{*,\gamma}&(I_{\Gamma}\otimes\gamma)\theta_A
				\end{bmatrix}\begin{bmatrix}
					D_{*,\gamma}D_C&0\\
					-\underline{A}^*\gamma^* D_C&D_A
				\end{bmatrix}\begin{bmatrix}
					0\\
					\iota^{\clh_A}_j(h_a)
				\end{bmatrix}\\
				=&\begin{bmatrix}
					D_{*,\gamma}&(I_{\Gamma}\otimes\gamma)\theta_A
				\end{bmatrix}
				\begin{bmatrix}
					0\\
					D_A \iota^{\clh_A}_j(h_a)
				\end{bmatrix}\\
				=&(I_{\Gamma}\otimes\gamma)\theta_A  D_A \iota^{\clh_A}_j(h_a)\\
				=&\theta_{C,E} D_E \iota^{\clh_E}_j(h_a).
			\end{align*}
			For $h_c\in \clh_C$  and $j=1, \ldots, d$, we have
			\begin{align*}
				\theta_{C,E} D_E \iota^{\clh_E}_j(h_c)
				&= D_C \iota^{\clh_C}_j(h_c)- \gamma D_{*,A}\underline{B} \iota^{\clh_C}_j(h_c)-\underset{|\alpha|\geq 1}{\sum}e_{\alpha}\otimes\gamma D_{*,A}A_{\alpha}^*\underline{B} \iota^{\clh_C}_j(h_c)\\
				&= D_C \iota^{\clh_C}_j(h_c)- \gamma D_{*,A}D_{*,A}\gamma^*D_C \iota^{\clh_C}_j(h_c)-\underset{|\alpha|\geq 1}{\sum}e_{\alpha}\otimes\gamma D_{*,A}A_{\alpha}^*D_{*,A}\gamma^*D_C \iota^{\clh_C}_j(h_c)\\
				&= D_C \iota^{\clh_C}_j(h_c)- \gamma (I-\underline{A}\underline{A}^*)\gamma^*D_C \iota^{\clh_C}_j(h_c)-\underset{|\alpha|\geq 1}{\sum}e_{\alpha}\otimes\gamma D_{*,A}A_{\alpha}^*D_{*,A}\gamma^*D_C \iota^{\clh_C}_j(h_c)\\
				&= (I-\gamma\gamma^*)D_C \iota^{\clh_C}_j(h_c)\\
				&\, \, \, \, -\left[- \gamma \underline{A}\underline{A}^*\gamma^*D_C \iota^{\clh_C}_j(h_c) +\underset{|\alpha|\geq 1}{\sum}e_{\alpha}\otimes\gamma D_{*,A}A_{\alpha}^*D_{*,A}\gamma^*D_C \iota^{\clh_C}_j(h_c)\right]\\
				&= D_{*,\gamma}^2D_C \iota^{\clh_C}_j(h_c)-\left[- \gamma \underline{A}+\underset{|\alpha|\geq 1}{\sum}e_{\alpha}\otimes\gamma D_{*,A}A_{\alpha-1}^*D_A \right](e_0\otimes\underline{A}^*\gamma^*D_C \iota^{\clh_C}_j(h_c))\\
				&= D_{*,\gamma}^2D_C \iota^{\clh_C}_j(h_c)-(I_{\Gamma}\otimes \gamma) \theta_A\underline{A}^*\gamma^*D_C \iota^{\clh_C}_j(h_c).
			\end{align*}
			Therefore, we obtain
			\begin{align*}
				\begin{bmatrix}
					D_{*,\gamma}&(I_{\Gamma}\otimes\gamma)\theta_A\end{bmatrix}{\sigma} D_E \iota^{\clh_E}_j(h_c)=&\begin{bmatrix}
					D_{*,\gamma}&(I_{\Gamma}\otimes\gamma)\theta_A	\end{bmatrix}\begin{bmatrix}
					D_{*,\gamma}D_C&0\\
					-\underline{A}^*\gamma^*D_C&D_A
				\end{bmatrix}
				\begin{bmatrix}
					\iota^{\clh_C}_j(h_c)\\0
				\end{bmatrix}\\
				=&\begin{bmatrix}
					D_{*,\gamma}&(I_{\Gamma}\otimes\gamma)\theta_A	\end{bmatrix}\begin{bmatrix}
					D_{*,\gamma}D_C \iota^{\clh_C}_j(h_c)\\
					-\underline{A}^*\gamma^*D_C \iota^{\clh_C}_j(h_c)
				\end{bmatrix}\\
				=& D_{*,\gamma}^2D_C \iota^{\clh_C}_j(h_c)-(I_{\Gamma}\otimes\gamma)\theta_A \underline{A}^*\gamma^*D_C \iota^{\clh_C}_j(h_c)\\
				=&\theta_{C,E} D_E \iota^{\clh_E}_j(h_c).
			\end{align*}
			This completes the proof.
		\end{proof}
		The following result illustrates that the characteristic function $\theta_{C,E}$ can also be viewed as transfer function for some colligation matrix.
		\begin{lemma}\label{lemma characteristic fun of colligation}
			Let $\underline{C}=(C_1,C_2,\ldots, C_d)$ be a row contraction on a Hilbert space $\clh_C$ and $\underline{E}=(E_1,E_2,\ldots, E_d)$ be a minimal contractive lifting of $\underline{C}$ on a Hilbert space $\clh_E=\clh_C\oplus \clh_A$, where
			\[E_i=\begin{bmatrix}
				C_i&0\\
				B_i&A_i
			\end{bmatrix}\text{ for }1\leq i\leq d.\] 
			Let $\gamma$ and $\sigma$ be as in Lemma \ref{eqn for sigma}. Then the characteristic function $\theta_{C,E}$ of the lifting $\underline{E}$ of $\underline{C}$ can be realised as the {transfer} function of the following colligation matrix:
			\[
			V=\begin{bmatrix}
				\underline{A}^*& D_AP_{D_A}\sigma\\
				\gamma D_{*,A}& (D_{*,\gamma}P_{D_{*,\gamma}}\sigma-\gamma \underline{A}P_{D_A}\sigma)
			\end{bmatrix}: \clh_A \oplus \mathcal{D}_E \to \clh^d_A \oplus \mathcal{D}_C.\]
				\end{lemma}
				\begin{proof}
					By Lemma \ref{characteristic function lemma}, note that 
					\begin{align*}
						\theta_{C,E}=&\begin{bmatrix}
							D_{*,\gamma}&(I_{\Gamma}\otimes\gamma) \theta_A
						\end{bmatrix}{\begin{bmatrix}
								P_{D_{*,\gamma}}\sigma\\
								P_{D_A}\sigma
						\end{bmatrix}}\\
						=& D_{*,\gamma}P_{D_{*,\gamma}}\sigma+(I_{\Gamma}\otimes\gamma)\theta_A{ P_{D_A}\sigma}
					\end{align*}
					{Thus, the Fourier expansion of $\theta_{C,E}$ is }
					\begin{align*}
						\theta_{C,E}=& D_{*,\gamma}P_{D_{*,\gamma}}\sigma-\gamma \underline{A}P_{D_A}\sigma+(I_{\Gamma}\otimes\gamma D_{*,A})({I_{\Gamma\otimes \clh_A}}-(\underline{R}\otimes I_{\Gamma\otimes \clh_A})(I_{\Gamma} \otimes \underline{A}^*))^{-1}\\
						&(\underline{R} \otimes I_{\clh_A})(I_{\Gamma} \otimes D_AP_{D_A}\sigma).
					\end{align*}
					We know that $$ D_{*,\gamma}P_{D_{*,\gamma}}\sigma-\gamma \underline{A}P_{D_A}\sigma+(I_{\Gamma}\otimes\gamma D_{*,A})({I_{\Gamma\otimes \clh_A}}-(\underline{R}  \otimes I_{\Gamma\otimes \clh_A})(I_{\Gamma} \otimes \underline{A}^*))^{-1}(\underline{R} \otimes I_{\clh_A})(I_{\Gamma} \otimes D_AP_{D_A}\sigma)$$ is the transfer function of the colligation
					{\begin{align*}
							V=\begin{bmatrix}
								\underline{A}^*& D_AP_{D_A}\sigma\\
								\gamma D_{*,A}& (D_{*,\gamma}P_{D_{*,\gamma}}\sigma-\gamma \underline{A}P_{D_A}\sigma)
							\end{bmatrix}.
							\end{align*}}
							This proves the result.
						\end{proof}
						
						Now, we observe some interesting facts about the colligation matrix obtained in Lemma \ref{lemma characteristic fun of colligation}.
						\begin{lemma}\label{lemma co-isometric colligation}
							Let $\underline{C},\underline{A},\underline{E}, \gamma$ and $\sigma$ be as defined in Lemma \ref{lemma characteristic fun of colligation}. Then the colligation matrix
							{\[
								V=\begin{bmatrix}
									\underline{A}^*& D_AP_{D_A}\sigma\\
									\gamma D_{*,A}& (D_{*,\gamma}P_{D_{*,\gamma}}\sigma-\gamma \underline{A}P_{D_A}\sigma)
								\end{bmatrix}.\]}
									is observable and co-isometric.
								\end{lemma}
								\begin{proof}
									Note that 
									\begin{align*}
										V=&\begin{bmatrix}
											\underline{A}^*& D_AP_{D_A}\sigma\\
											\gamma D_{*,A}& (D_{*,\gamma}P_{D_{*,\gamma}}\sigma-\gamma \underline{A}P_{D_A}\sigma)
										\end{bmatrix}\\
										=&\begin{bmatrix}
											\underline{A}^*&0& D_A\\
											\gamma D_{*,A}& D_{*,\gamma}&- \gamma \underline{A}
										\end{bmatrix}\begin{bmatrix}
											I_{ \clh_A}&0\\
											0& P_{D_{*,\gamma}}\sigma\\
											0& P_{D_A}\sigma
										\end{bmatrix}\\
										=&\begin{bmatrix}
											\underline{A}^*&0& D_A\\
											\gamma D_{*,A}& D_{*,\gamma}&- \gamma \underline{A}
										\end{bmatrix}\begin{bmatrix}
											I_{ \clh_A}&0\\
											0& \sigma
										\end{bmatrix}.
									\end{align*}
									Then
									{\begin{align*}
											VV^*=&\begin{bmatrix}
												\underline{A}^*&0& D_A\\
												\gamma D_{*,A}& D_{*,\gamma}&- \gamma \underline{A}
											\end{bmatrix}\begin{bmatrix}
												I_{ \clh_A}&0\\
												0& \sigma
											\end{bmatrix}\begin{bmatrix}
												I_{ \clh_A}&0\\
												0& \sigma^*
											\end{bmatrix}\begin{bmatrix}
												\underline{A}& D_{*,A}\gamma^*  \\
												0& D_{*,\gamma}\\
												D_A&-  \underline{A}^*\gamma^*
											\end{bmatrix}\\
											=&\begin{bmatrix}
												\underline{A}^*&0& D_A\\
												\gamma D_{*,A}& D_{*,\gamma}&- \gamma \underline{A}
											\end{bmatrix}\begin{bmatrix}
												\underline{A}& D_{*,A}\gamma^*  \\
												0& D_{*,\gamma}\\
												D_A&-  \underline{A}^*\gamma^*
											\end{bmatrix}\\
											=&\begin{bmatrix}
												(\underline{A}^*\underline{A}+D_A^2)&(\underline{A}^*D_{*,A}\gamma^*-D_A\underline{A}^*\gamma^*)\\
												(\gamma D_{*,A}\underline{A}-\gamma\underline{A}D_A)&(\gamma D_{*,A}^2\gamma^*+\gamma \underline{A}\underline{A}^*\gamma^*+D_{*,\gamma}^2)
											\end{bmatrix}\\
											=&\begin{bmatrix}
												I_{ \clh_A}&0\\
												0&I_{ \clh_C}
											\end{bmatrix}.
									\end{align*}}
									This implies $V$ is co-isometric.
									
									To show $V$ is observable, we need to show that $\mathcal{O}_{\gamma D_{*,A},A^*}$ is injective, {where $\mathcal{O}_{\gamma D_{*,A},A^*}:\clh_A\rightarrow\Gamma\otimes\mathcal{D}_C$ is defined by $\mathcal{O}_{\gamma D_{*,A},A^*}x=\gamma D_{*,A}\left(\underset{\alpha}{\sum}e_{\alpha}\otimes A_{\alpha}^*x\right)$}. Suppose $\mathcal{O}_{\gamma D_{*,A},A^*}( x)=0$ for some $ x\in \clh_A$. Then
									{\begin{align*}
											\gamma D_{*,A}\left(\underset{\alpha}{\sum}e_{\alpha}\otimes A_{\alpha}^* x\right)=0.
									\end{align*}}
									This is equivalent to say that {$$\underset{\alpha}{\sum}e_{\alpha}\otimes(\gamma D_{*,A}A_{\alpha}^*x)=0.$$ 
										As a result, $\gamma D_{*,A}A_{\alpha}^*x=0$ for every $\alpha\in \tilde{\Lambda}$.} We know that $\gamma$ is resolving, thus $\gamma D_{*,A}A_{\alpha}^*x=0$ for every $\alpha\in{\tilde{\Lambda}}$ implies $D_{*,A}A_{\alpha}^*x=0$ for every $\alpha$. In particular, if we take $|\alpha|=0$, then $D_{*,A}x=0$. Because $A$ is a c.n.c. row contraction,  $\mathcal{O}_{\gamma D_{*,A},A^*}$ is an injective operator (see Lemma \ref{lemma sec 3}).
								\end{proof}
								\section{Characteristic functions of liftings transformed by a Blaschke factor}
								If $T$ is a contraction on  a Hilbert space $\clh$, then for any $a$ in the open unit disc $\mathbb{D}$, the operator $T_a := (T-aI)(I-\overline{a}T)^{-1}$ is also a contraction on $\clh$. Define the contraction
								\begin{equation}\label{eqn of S_T}
									S_T :=(1-|a|^2)^{1/2}(I-\overline{a}T)^{-1}.
								\end{equation}  According to equations $(1.8)$ on page 246 of \cite{NAGY}, there exist unitaries $Z_T:\cld_{T_a}\rightarrow\cld_{T}$ and $Z_{*,T}:\cld_{*,T_a}\rightarrow\cld_{*,T}$ such that 
								\begin{align}\label{eqn for Z}
									Z_TD_{T_a} &= D_TS_T\\
									\label{eqn Z*}Z_{*,T}D_{*,T_a} &= D_{*,T}S^*_{T}.
								\end{align}
								Let $\theta_{T_a}$ and $\theta_{T}$ be the symbols of characteristic functions of $T_a$ and $T$ respectively. It is also shown on page 247 of \cite{NAGY} that 
								\begin{equation}\label{eqn for cf}
									Z_{*,T}\theta_{T_a}(\lambda)Z^{-1}_T = \theta_T(\mu)\text{ for } \lambda\in\mathbb{D},
								\end{equation}	
								where  $\mu = \displaystyle \frac{\lambda+a}{1+\overline{a}\lambda}.$ 
								\begin{theorem}\label{thm1} Let   
									$E = \begin{bmatrix}
										C&0\\
										B&A
									\end{bmatrix}$ be a minimal contractive lifting of $C$ on a Hilbert space $\clh_E.$ Then for $a\in\mathbb{D}$, the operator $E_a = (E-aI)(I-\overline{a}E)^{-1}$ is a minimal contractive lifting of $C_a$ and satisfies
									\[
									E_a = \begin{bmatrix}
										C_a &0\\
										S_ABS_C & A_a
									\end{bmatrix}
									\]
									where 
									$S_C = (1-|a|^2)^{1/2}(I-\overline{a}C)^{-1}$ and $S_A = (1-|a|^2)^{1/2}(I-\overline{a}A)^{-1}$.
								\end{theorem}
								\begin{proof}
									By Equations (\ref{eqn for Z}) and (\ref{eqn Z*}), there exist unitaries $Z_C: \cld_{C_a}\rightarrow\cld_C$, $Z_{*,C}: \cld_{*,C_a}\rightarrow\cld_{*,C}$, $Z_A: \cld_{A_a}\rightarrow\cld_A$ and $Z_{*,A}: \cld_{*,A_a}\rightarrow\cld_{*,A}$ such that
									\begin{align}\label{eqn11}
										Z_CD_{C_a} &= D_CS_C, \,\,\:\:\:Z_{*,C}D_{*,C_a} = D_{*,C}{S^*_C},\\
										Z_AD_{A_a} &= D_AS_A \:\:\text{and}\:\:Z_{*,A}D_{*,A_a} = D_{*,A}{S^*_A}.
									\end{align}
									Therefore
									\begin{align*}
										E_a &= (E-aI)(I-\overline{a}E)^{-1}\\
										&=\begin{bmatrix}
											C-aI&0\\B&A-aI
										\end{bmatrix}
										\begin{bmatrix}
											(I-\overline{a}C)^{-1}&0\\(I-\overline{a}A)^{-1}\overline{a}B(I - \overline{a}C)^{-1}&(I-\overline{a}A)^{-1}
										\end{bmatrix}\\
										&=\begin{bmatrix}
											(C-aI)(I-\overline{a}C)^{-1}&0\\B(I-\overline{a}C)^{-1}+(A-aI)(I-\overline{a}A)^{-1}\overline{a}B(I-\overline{a}C)^{-1}&(A-aI)(I-\overline{a}A)^{-1}
										\end{bmatrix}
										\\		&=\begin{bmatrix}
											C_a & 0\\\tilde{B}&A_a
										\end{bmatrix},
									\end{align*}
									where
									\begin{align*}
										\tilde{B} &= B(I-\overline{a}C)^{-1}+(A-aI)(I-\overline{a}A)^{-1}\overline{a}B(I-\overline{a}C)^{-1}\\
										&=[I+\overline{a}(A-aI)(I-\overline{a}A)^{-1}]B(I-\overline{a}C)^{-1}\\
										&=[(I-\overline{a}A)+\overline{a}(A-aI)](I-\overline{a}A)^{-1}B(I-\overline{a}C)^{-1}\\
										&=(I-\overline{a}A+\overline{a}A-|a|^2I)(I-\overline{a}A)^{-1}B(I-\overline{a}C)^{-1}\\
										&=(1-|a|^2)(I-\overline{a}A)^{-1}B(I-\overline{a}C)^{-1}\\
										&=S_ABS_C.
									\end{align*}
									Hence $E_a = \begin{bmatrix}
										C_a &0\\
										S_ABS_C & A_a
									\end{bmatrix}$ is a contractive lifting of $C_a$. Since $\left(\displaystyle\frac{\lambda-a}{1-\bar{a}\lambda}\right)^n$ has the Taylor series expansion $$\left(\frac{\lambda-a}{1-\bar{a}\lambda}\right)^n = \underset{v=0}{\overset{\infty}{\sum}}c_v(a,n)\lambda^v$$ with radius of convergence greater than 1, we get
									$${E^n_a}= \sum_{v=0}^{\infty}c_v(a,n)E^v.$$
									This implies
									\begin{align*}
										\bigvee_{n=0}^\infty E^n_a&{\clh_C}\subset \bigvee_{n=0}^\infty E^n{\clh_C}
									\end{align*}	
where $\displaystyle \bigvee_{n=0}^\infty$ stand for closed linear span.						
									Also, using the fact that $({E_a})_{-a} = E$ for  $a\in\mathbb{D},$ we obtain $\displaystyle \bigvee_{n=0}^\infty E^n_a\clh_C= \bigvee_{n=0}^\infty E^n{\clh_C}=\clh_{E}$. Hence $E_a$ is a minimal contractive lifting of $C_a$.  
								\end{proof}

								\begin{theorem} \label{thm2}
									Let $C$ be a contraction on a Hilbert space $\clh_C.$ Let  $E = \begin{bmatrix}
										C&0\\
										B&A
									\end{bmatrix}$ be a minimal contractive lifting
									of $C$. Then for $a\in\mathbb{D}$ the symbol $\theta_{C_a,E_a}(\lambda)$ of the characteristic function of lifting $E_a$ coincides with $\theta_{C,E}(\mu)$, where $\mu= \displaystyle\frac{\lambda+a}{1+\overline{a}\lambda}$. 
								\end{theorem}
								\begin{proof}
									By Theorem \ref{thm1}, we know that $E_a$ is a minimal contractive lifting of $C_a$. Thus
									there is a contraction $\gamma_a:\cld_{*,A_a}\rightarrow\cld_{C_a}$ such that $\tilde{B} = D_{*,A_a}\gamma^*_aD_{C_a}$. 
									\begin{align*}
										\tilde{B}	&= D_{*,A_a}\gamma^*_aD_{C_a}=S_ABS_C\\&=S_AD_{*,A}\gamma^*D_CS_C\end{align*}
									Using Equations (\ref{eqn for Z}) and (\ref{eqn Z*}), we obtain
									\[\tilde{B}	=D_{*,A_a}Z^*_{*,A}\gamma^*Z_CD_{C_a}.
									\]
									
									Since $A_a$ is a c.n.c. contraction, we have 	$\gamma^*_aD_{C_a} = Z^*_{*,A}\gamma^*Z_CD_{C_a}$ and thus
									\begin{align}\label{eqn3}
										\gamma_a& = Z^*_C\gamma Z_{*,A}\\\label{eqn3.1}
										D_{*,\gamma_a} &= Z^*_{C}D_{*,\gamma}Z_C.
									\end{align}
									By Equations (\ref{eqn for cf}), (\ref{eqn3}) and (\ref{eqn3.1}) we have
									\begin{align*}
										\theta_{C_a,E_a}(\lambda)D_{E_a}&=\begin{bmatrix}
											D_{*,\gamma_a}&\gamma_a\theta_{A_a}(\lambda)\end{bmatrix}
										\begin{bmatrix}
											D_{*,\gamma_a}D_{C_a}&0\\-A^*_a\gamma^*_aD_{C_a}&D_{A_a}\end{bmatrix}\\
										&=\begin{bmatrix}
											Z^*_CD_{*,\gamma}Z_C&Z^*_{C}\gamma Z_{*,A}Z^*_{*,A}\theta_A(\mu)Z_A
										\end{bmatrix}\begin{bmatrix}
											D_{*,\gamma_a}D_{C_a}&0\\
											-A^*_a\gamma^*_aD_{C_a}&D_{A_a}
										\end{bmatrix}\\
										&=\begin{bmatrix}
											Z^*_CD_{*,\gamma}Z_C&Z^*_{C}\gamma\theta_A(\mu)Z_A
										\end{bmatrix}\begin{bmatrix}
											D_{*,\gamma_a}D_{C_a}&0\\
											-A^*_a\gamma^*_aD_{C_a}&D_{A_a}
										\end{bmatrix}\\
										&=Z^*_C\begin{bmatrix}
											D_{*,\gamma}&\gamma\theta_A(\mu)
										\end{bmatrix}
										\begin{bmatrix}
											Z_C&0\\
											0&Z_A
										\end{bmatrix}\begin{bmatrix}
											D_{*,\gamma_a}D_{C_a}&0\\
											-A^*_a\gamma^*_aD_{C_a}&D_{A_a}
										\end{bmatrix}
										\\&= Z^*_{C}\begin{bmatrix}
											D_{*,\gamma} & \gamma\theta_A(\mu)
										\end{bmatrix}\sigma\sigma^{-1}
										\begin{bmatrix}
											Z_C & 0\\ 0 & Z_A
										\end{bmatrix}\begin{bmatrix}
											D_{*,\gamma_a}D_{C_a}&0\\
											-A_a^*\gamma_a^*D_{C_a}&D_{A_a}
										\end{bmatrix}\\
										&= Z^*_{C}\theta_{C,E}(\mu){\sigma}^{-1}
										\begin{bmatrix}
											Z_C & 0\\ 0 & Z_A
										\end{bmatrix}\begin{bmatrix}
											D_{*,\gamma_a}D_{C_a}&0\\
											-A_a^*\gamma_a^*D_{C_a}&D_{A_a},
										\end{bmatrix}
									\end{align*}
									where the second equality follows using Equation (\ref{eqn11}). Note that by Equation (\ref{eqn3}), $D_{*,\gamma_a} = Z^*_{C}D_{*,\gamma}Z_C$ and consequently 
									\begin{equation}\label{eqnZ_C}
										Z_CD_{*,\gamma_a} = D_{*,\gamma}Z_C.
									\end{equation}
									The maps $\sigma_aD_{E_a} = \begin{bmatrix}
										D_{*,\gamma_a}D_{C_a}&0\\
										-A_a^*\gamma_a^*D_{C_a}&D_{A_a}
									\end{bmatrix}$ and $\sigma D_E = \begin{bmatrix}
										D_{*,\gamma}D_{C}&0\\
										-A^*\gamma^*D_{C}&D_{A}
									\end{bmatrix}$ are unitary maps from $\cld_{E_a}$ onto $\cld_{*,\gamma_a}\oplus\cld_{A_a}$ and from $\cld_{E}$ onto $\cld_{*,\gamma}\oplus\cld_{A}$ respectively. Thus, $$\sigma^{-1}\begin{bmatrix}
										Z_C & 0\\
										0&Z_A
									\end{bmatrix}\sigma_a$$ is a unitary maps from $\cld_{E_a}$ onto $\cld_E$. Then
									\begin{align*}
										\theta_{C_a,E_a}(\lambda)D_{E_a} &= Z^*_{C}\theta_{C,E}(\mu){\sigma}^{-1}
										\begin{bmatrix}
											Z_C & 0\\ 0 & Z_A
										\end{bmatrix}\sigma_aD_{E_a}.
									\end{align*}
									Hence,  $\theta_{C_a,E_a}(\lambda)$ coincides with $\theta_{C,E}(\mu).$
								\end{proof}	
								
								\begin{theorem}
									Let $C$ be a contraction on a Hilbert space $\clh_C.$ Let $E,a$ and $E_a$ be as defined in Theorem \ref{thm2}. Let $\gamma:\mathcal{D}_{*,A}\rightarrow \mathcal{D}_C$ be a contraction such that $B^*=D_C\gamma D_{*,A}.$ Then the characteristic function of lifting $E_a$ satisfies the following relation 
									$$\theta_{C_a,E_a}(\lambda)D_{E_a} = Z^*_C\begin{bmatrix}
										D_{*,\gamma}&\gamma\theta_A(\mu)
									\end{bmatrix}
									\begin{bmatrix}
										D_{*,\gamma}D_C&0\\
										{\theta_A(a)}^*\gamma^*D_C&D_{A_a}
									\end{bmatrix}\begin{bmatrix}
										S_C&0\\0&S_A\end{bmatrix} \text{ for }\lambda\in\mathbb{D},$$
									where $S_C$ and $S_A$ are the contractions as given in Equation \ref{eqn of S_T}.
								\end{theorem}
								\begin{proof}
									From the proof of  Theorem \ref{thm2}, we have \begin{align*}
										\theta_{C_a,E_a}(\lambda)D_{E_a} &= Z^*_{C}\begin{bmatrix}
											D_{*,\gamma} & \gamma\theta_A(\mu)
										\end{bmatrix}
										\begin{bmatrix}
											Z_C & 0\\ 0 & Z_A
										\end{bmatrix}\begin{bmatrix}
											D_{*,\gamma_a}D_{C_a}&0\\
											-A_a^*\gamma_a^*D_{C_a}&D_{A_a}
										\end{bmatrix}	\\
										&=Z^*_C\begin{bmatrix}
											D_{*,\gamma}&\gamma\theta_A(\mu)
										\end{bmatrix}
										\begin{bmatrix}
											Z_CD_{*,\gamma_a}D_{C_a}&0\\
											-Z_AA_a^*\gamma_a^*D_{C_a}&Z_AD_{A_a}
										\end{bmatrix}.\end{align*}
									Using the identity $\theta_{A_a}(0)=A_a$ and  Equation (\ref{eqnZ_C}), we rewrite $\theta_{C_a,E_a}(\lambda)D_{E_a}$  as
									$$\theta_{C_a,E_a}(\lambda)D_{E_a}=Z^*_C\begin{bmatrix}
										D_{*,\gamma}&\gamma\theta_A(\mu)
									\end{bmatrix}\begin{bmatrix}
										D_{*,\gamma}Z_CD_{C_a}&0\\
										Z_A(\theta_{A_a}(0))^*\gamma^*_aD_{C_a}&D_{A_a}S_A
									\end{bmatrix}.$$ Now by Equation (\ref{eqn for cf}), we obtain
									$$\theta_{C_a,E_a}(\lambda)D_{E_a}=Z^*_C\begin{bmatrix}
										D_{*,\gamma}&\gamma\theta_A(\mu)
									\end{bmatrix}
									\begin{bmatrix}
										D_{*,\gamma}D_CS_C&0\\
										{\theta_A(a)}^*Z_{*,A}\gamma^*_aD_{C_a}&D_{A_a}S_A
									\end{bmatrix}.$$
									Further, by using Equations (\ref{eqn3}) and (\ref{eqn11}), we get
									\begin{align*}
										\theta_{C_a,E_a}(\lambda)D_{E_a}&=Z^*_C\begin{bmatrix}
											D_{*,\gamma}&\gamma\theta_A(\mu)
										\end{bmatrix}
										\begin{bmatrix}
											D_{*,\gamma}D_CS_C&0\\
											{\theta_A(a)}^*\gamma^*Z_CD_{C_a}&D_{A_a}S_A
										\end{bmatrix}\\
										&=Z^*_C\begin{bmatrix}
											D_{*,\gamma}&\gamma\theta_A(\mu)
										\end{bmatrix}
										\begin{bmatrix}
											D_{*,\gamma}D_CS_C&0\\
											{\theta_A(a)}^*\gamma^*D_CS_C&D_{A_a}S_A
										\end{bmatrix}\\
										&=Z^*_C\begin{bmatrix}
											D_{*,\gamma}&\gamma\theta_A(\mu)
										\end{bmatrix}
										\begin{bmatrix}
											D_{*,\gamma}D_C&0\\
											{\theta_A(a)}^*\gamma^*D_C&D_{A_a}
										\end{bmatrix}\begin{bmatrix}
											S_C&0\\0&S_A
										\end{bmatrix}.
									\end{align*}
									This completes the proof.
								\end{proof} 
								Let $A$ be a contraction on a Hilbert space $\clh_A$ and $\mathbb{D}$ denote the open unit disk. For each $\lambda\in\mathbb{D},$ we can define a contraction $A_{\lambda} := (A-\lambda I)(I-\bar{\lambda}A)^{-1}$ on $\clh_A$. If $\theta_{A_{\lambda}}$ is the characteristic function of $A_{\lambda}$ then
								\begin{equation}
									\| \theta_{A_{\lambda}}(0) \|= \| \theta_{A}(\lambda) \|
								\end{equation}
								(cf. Section 1.3 of Chapter VII of \cite{NAGY}). The following theorem establishes a relationship between the norm of the characteristic function of a lifting and the norm of the contraction $A_{\lambda}$.
								
								\begin{theorem}
									Let $E = \begin{bmatrix}
										C&0\\
										B&A
									\end{bmatrix}$ be a minimal contractive lifting of $C$ and $\lambda\in\mathbb{D},$ then 
									\[
									\norm{\theta_{C,E}(\lambda)}\leq \norm{D_{*,\gamma}}+\norm{(A-\lambda I)(I-\bar{\lambda}A)^{-1}}.
									\]
								\end{theorem} 
								\begin{proof} 
									By Proposition \ref{Prop lifting}, we know that there exists a contraction $\gamma:\cld_{*,A}\rightarrow\cld_{C}$ such that $B = D_{*,A}\gamma^*D_C$. The characteristic function $\theta_{C,E}$ for the lifting E of C is given by    
									\[\theta_{C,E}(\lambda) = \begin{bmatrix}
										D_{*,\gamma}&\gamma\theta_{A}(\lambda)
									\end{bmatrix}\sigma\] where $\sigma$ is a unitary map from $\cld_{E}$ to $\cld_{*,\gamma}\oplus\cld_A$. Thus 
									\begin{align*}
										\norm{\theta_{C,E}(\lambda)}&=||\begin{bmatrix}
											D_{*,\gamma}&\gamma\theta_{A}(\lambda)
										\end{bmatrix}\sigma||\\
										&\leq \norm{D_{*,\gamma}}+\norm{\gamma\theta_{A}({\lambda})}\\
										&\leq \norm{D_{*,\gamma}}+\norm{\theta_{A}({\lambda})}\\
										&=\norm{D_{*,\gamma}}+\norm{\theta_{A_{\lambda}}(0)}\\
										&\leq \norm{D_{*,\gamma}}+\norm{A_{\lambda}}\\
										&=\norm{D_{*,\gamma}}+\norm{(A-\lambda I)(I-\bar{\lambda}A)^{-1}}.
									\end{align*}
									In the second last step of the above computation, we have used $\theta_{A_{\lambda}}(0) = A_{\lambda}|_{\cld_{A}}.$This completes the proof.
								\end{proof}
								
								\section{Equivalence of lifting of contractions}
								In this section, we discuss unitary equivalence of two minimal contractive liftings of a row contraction and unitary equivalence of their  characteristic function. First we recall the definition of equivalence of two multi-analytic operators, which we will use throughout this section.
								\begin{definition}\cite{DEYII}
									We say two multi-analytic operators $M_1:\Gamma\otimes \clh_1\rightarrow\Gamma\otimes \clh$ and $M_2:\Gamma\otimes \clh_2\rightarrow \Gamma\otimes \clh$ with symbols $\theta_1$ and $\theta_2$ are {\it equivalent}, if there exists a unitary $v:\clh_1\rightarrow \clh_2$ such that $\theta_1=\theta_2v$. We also say $\theta_1$ and $\theta_2$ are {\it equivalent} and write $\theta_1 \simeq \theta_2.$
								\end{definition}

								In this section, one of our main result is the following.
								\begin{theorem}
									Let $\underline{C}$ be a row contraction and $\underline{E}$, $\underline{E}'$ be two minimal contractive liftings of $\underline{C}$. Then $\underline{E}$ and $\underline{E}'$ are unitarily equivalent if and only if $\theta_{C,E}$ and $\theta_{C,E'}$ are equivalent.
								\end{theorem}
								We have divided the proof of above theorem into the following two result. 
								
								The following result was proved in \cite{DEY}, but here we present a simple proof of this using colligation matrix.
								\begin{proposition}\label{prop reduces lifting equivalent}
									Let $\underline{C}=(C_1,C_2,\ldots, C_d)$ be a row contraction on a Hilbert space $\clh_C$. Let
									\[\underline{E}=\begin{bmatrix}
										\underline{C}&0\\
										\underline{B}&\underline{A}
									\end{bmatrix}  \mbox{~and~} \underline{E}'=\begin{bmatrix}
										\underline{C}&0\\
										\underline{B}'&\underline{A}'
									\end{bmatrix}\]  be two  minimal contractive liftings of $\underline{C}$ on Hilbert spaces $\clh_E=\clh_C\oplus \clh_A$ and $\clh_{E^{'}}=\clh_C\oplus \clh_{A^{'}}$ respectively. If  $U:\clh_A\rightarrow \clh_{A^{'}}$ is  a unitary such that
									\begin{align*}
										\begin{bmatrix}
											I_{\clh_C}&0\\
											0&U
										\end{bmatrix}	\begin{bmatrix}
											\underline{C}&0\\
											\underline{B}&\underline{A}
										\end{bmatrix}=
										\begin{bmatrix}
											\underline{C}&0\\
											\underline{B}'&\underline{A}'
										\end{bmatrix}\begin{bmatrix}
											I_{\underset{i=1}{\overset{d}{\oplus}}\clh_C}&0\\
											0&\underset{i=1}{\overset{d}{\oplus}}U
										\end{bmatrix},
									\end{align*}
then $\theta_{C,E}$ and $\theta_{C,E^{'}}$ are equivalent.
								\end{proposition}
								\begin{proof} The unitary $U:\clh_A\rightarrow \clh_{A^{'}}$ satisfies
									\begin{align}\label{eqn1}
										UB_i=B_i^{'}\text{ and }UA_i=A_i^{'}U.
									\end{align}
									On computing $I-\underline{E}^*\underline{E}$ and using the above equation, we obtain
									\begin{align*}
										D_E=\begin{bmatrix}
											I_{\clh_C}&0\\
											0&\underset{i=1}{\overset{d}{\oplus}}U^*
										\end{bmatrix}D_{E'}\begin{bmatrix}
											I_{\underset{i=1}{\overset{d}{\oplus}}\clh_C}&0\\
											0&\underset{i=1}{\overset{d}{\oplus}}U
										\end{bmatrix}.
									\end{align*}
									Similarly, evaluating $I-\underline{A}\underline{A}^*,\,I-\underline{A}^*\underline{A}$ and using Equation (\ref{eqn1}), we get
									\begin{align}\label{eqn2}
										D_A=&\left(\underset{i=1}{\overset{d}{\oplus}}U^*\right)D_{A^{'}}\left(\underset{i=1}{\overset{d}{\oplus}}U\right),\\
										\label{eqn4}	D_{*,A}=&U^*D_{*,A^{'}}U.
									\end{align}
									From Equations (\ref{eqn1}), (\ref{eqn2}) and the representations $\underline{B}=D_{*,A}\gamma^*D_C$, $\underline{B}^{'}=D_{*,A^{'}}{\gamma'}^*D_C$ (from Proposition \ref{Prop lifting}), we have
									\begin{align*}
										D_{*,A^{'}}{\gamma'}^*D_C=&	UD_{*,A}\gamma^*D_C\\
										=&D_{*,A^{'}}U\gamma^*D_C.
									\end{align*}
									Using the fact that $\underline{A}^{'}$ is a c.n.c. row contraction, we get $U\gamma^*={\gamma'}^*$, and consequently
									\begin{align}\label{eqn 3}
										\gamma=\gamma 'U|_{\cld_{*,A}}.
									\end{align}
									Thus, $D_{*,\gamma}=D_{*,\gamma'}.$
									From Equations (\ref{eqn1}) and (\ref{eqn 3}),
									$\underline{A}^*\gamma^*D_C=\left(\underset{i=1}{\overset{d}{\oplus}}U^*\right)\underline{A^{'}}^*{\gamma '}^*D_C$. Let $\sigma$ be the unitary operator defined in Lemma \ref{characteristic function lemma}. Then
									\begin{align*}
										\sigma D_E=&\begin{bmatrix}
											D_{*,\gamma}D_C&0\\
											-\underline{A}^*\gamma^*D_C&D_A
										\end{bmatrix}\\
										=&\begin{bmatrix}
											D_{*,\gamma'}D_C&0\\
											-\left(\underset{i=1}{\overset{d}{\oplus}}U^*\right)\underline{A'}^*{\gamma'}^*D_C&\left(\underset{i=1}{\overset{d}{\oplus}}U^*\right)D_{A'}\left(\underset{i=1}{\overset{d}{\oplus}}U\right)
										\end{bmatrix}\\
										=&\begin{bmatrix}
											I_{\underset{i=1}{\overset{d}{\oplus}}\clh_C}&0\\
											0&\underset{i=1}{\overset{d}{\oplus}}U^*
										\end{bmatrix}
										\begin{bmatrix}
											D_{*,\gamma '}D_C&0\\
											-\underline{A^{'}}^*{\gamma '}^*D_C&D_{A^{'}}\left(\underset{i=1}{\overset{d}{\oplus}}U\right)
										\end{bmatrix}\\
										=&\begin{bmatrix}
											I_{\underset{i=1}{\overset{d}{\oplus}}\clh_C}&0\\
											0&\underset{i=1}{\overset{d}{\oplus}}U^*
										\end{bmatrix}
										\begin{bmatrix}
											D_{*,\gamma '}D_C&0\\
											-\underline{A^{'}}^*{\gamma '}^*D_C&D_{A^{'}}
										\end{bmatrix}
										\begin{bmatrix}
											I_{\underset{i=1}{\overset{d}{\oplus}}\clh_C}&0\\
											0&\left(\underset{i=1}{\overset{d}{\oplus}}U\right)
										\end{bmatrix}\\
										=&\begin{bmatrix}
											I_{\underset{i=1}{\overset{d}{\oplus}}\clh_C}&0\\
											0&\underset{i=1}{\overset{d}{\oplus}}U^*
										\end{bmatrix}
										\sigma^{'}D_{E^{'}}
										\begin{bmatrix}
											I_{\underset{i=1}{\overset{d}{\oplus}}\clh_C}&0\\
											0&\left(\underset{i=1}{\overset{d}{\oplus}}U\right)
										\end{bmatrix}\\
										=&\begin{bmatrix}
											I_{\underset{i=1}{\overset{d}{\oplus}}\clh_C}&0\\
											0&\underset{i=1}{\overset{d}{\oplus}}U^*
										\end{bmatrix}
										\sigma^{'}
										\begin{bmatrix}
											I_{\underset{i=1}{\overset{d}{\oplus}}\clh_C}&0\\
											0&\left(\underset{i=1}{\overset{d}{\oplus}}U\right)
										\end{bmatrix}
										D_E.
									\end{align*}
									Since $UA_i=A_i^{'}U$, we have $(I_{\Gamma}\otimes U|_{D_{*,A}})\theta_A=\theta_{A^{'}}\left(\underset{i=1}{\overset{d}{\oplus}}U\right)|_{\cld_A}$ by Remark \ref{Remark popescu charac fn}.
									Thus
									\begin{align*}
										\theta_{C,E}=&\begin{bmatrix}
											D_{*,\gamma}&\left(I_{\Gamma}\otimes\gamma\right)\theta_{A}
										\end{bmatrix}\sigma\\
										=&\begin{bmatrix}
											D_{*,\gamma'}&\left(I_{\Gamma}\otimes\gamma'\right)(I_{\Gamma}\otimes U)\theta_{A}
										\end{bmatrix}\begin{bmatrix}
											I_{\underset{i=1}{\overset{d}{\oplus}}\clh_C}&0\\
											0&\underset{i=1}{\overset{d}{\oplus}}U^*
										\end{bmatrix}
										\sigma^{'}
										\begin{bmatrix}
											I_{\underset{i=1}{\overset{d}{\oplus}}\clh_C}&0\\
											0&\left(\underset{i=1}{\overset{d}{\oplus}}U\right)
										\end{bmatrix}|_{\cld_E}
	\end{align*}
\begin{align*}
										=&\begin{bmatrix}
											D_{*,\gamma'}&\left(I_{\Gamma}\otimes\gamma'\right)\theta_{A^{'}}\left(\underset{i=1}{\overset{d}{\oplus}}U\right)
										\end{bmatrix}\begin{bmatrix}
											I_{\underset{i=1}{\overset{d}{\oplus}}\clh_C}&0\\
											0&\underset{i=1}{\overset{d}{\oplus}}U^*
										\end{bmatrix}
										\sigma^{'}
										\begin{bmatrix}
											I_{\underset{i=1}{\overset{d}{\oplus}}\clh_C}&0\\
											0&\left(\underset{i=1}{\overset{d}{\oplus}}U\right)
										\end{bmatrix}|_{\cld_E}\\
										=&\begin{bmatrix}
											D_{*,\gamma'}&\left(I_{\Gamma}\otimes\gamma'\right)\theta_{A^{'}}
										\end{bmatrix}\sigma^{'}
										\begin{bmatrix}
											I_{\underset{i=1}{\overset{d}{\oplus}}\clh_C}&0\\
											0&\left(\underset{i=1}{\overset{d}{\oplus}}U\right)
										\end{bmatrix}|_{\cld_E}\\
										=&\theta_{C,E^{'}}\begin{bmatrix}
											I_{\underset{i=1}{\overset{d}{\oplus}}\clh_C}&0\\
											0&\left(\underset{i=1}{\overset{d}{\oplus}}U\right)
										\end{bmatrix}|_{\cld_E}.
									\end{align*}
									Hence $\theta_{C,E}$ and $\theta_{C,E^{'}}$ are equivalent.
								\end{proof}
								Now, we prove that if characteristic function of two minimal contractive  liftings are equivalent, then the liftings are unitarily equivalent. The proof uses  colligation matrix based techniques and not the functional model theory.
								\begin{proposition}\label{thm same characteristic function lifting}
									Let $\underline{C}=(C_1,C_2,\ldots, C_d)$ be a row contraction. Suppose  $\underline{E}$ and $\underline{E'}$ be two minimal contractive liftings of $\underline{C}$ such that characteristic functions $\theta_{C,E}$ and $\theta_{C,E'}$ of liftings are equivalent. Then $\underline{E}$ and $\underline{E'}$ are unitarily equivalent.
								\end{proposition}
								\begin{proof} Let 
									\[E_i=\begin{bmatrix}
										C_i&0\\
										B_i&A_i
									\end{bmatrix}  \mbox{~and~} E'_i=\begin{bmatrix}
										C_i&0\\
										B'_i&A'_i
									\end{bmatrix}\] 
									be two liftings of $\underline{C}$ on Hilbert spaces $\clh_E=\clh_C\oplus \clh_A$  and $\clh_{E'}=\clh_C\oplus \clh_{A'}$ respectively, for $i=1,2,\ldots,d$ such that $\theta_{C,E}$ and $\theta_{C,E'}$ are equivalent. So, there is a unitary $v:\mathcal{D}_E\rightarrow\mathcal{D}_{E'}$ such that {$\theta_{C,E}=\theta_{C,E'} v$}.
									
									Let $\gamma$ and $\gamma'$ be the contractions, and $\sigma$ and $\sigma'$ be the unitary operators associated with $\underline{E}$ and $\underline{E'}$ respectively, as in Lemma \ref{characteristic function lemma}. By Lemma \ref{lemma characteristic fun of colligation}, $\theta_{C,E}$ and {$\theta_{C,E'} v$} are {transfer} functions of the colligation matrices 
									{\begin{align*}
											W_1=&\begin{bmatrix}
												\underline{A}^*& D_AP_{D_A}\sigma\\
												\gamma D_{*,A}& \left( D_{*,\gamma}P_{D_{*,\gamma}}\sigma-\gamma \underline{A}P_{D_A}\sigma\right)
											\end{bmatrix}\text{ and }\\	
											W_2=&\begin{bmatrix}
												\underline{A'}^*& D_{A'}P_{D_{A'}}\sigma'v\\
												\gamma' D_{*,A'}& \left( D_{*,\gamma'}P_{D_{*,\gamma'}}\sigma'v-\gamma' \underline{A'}P_{D_{A'}}\sigma'v\right)
											\end{bmatrix},
									\end{align*}}
									respectively. By Lemma \ref{lemma ball} and Lemma \ref{lemma co-isometric colligation}, we know that $W_1$ and $W_2$ are unitarly equivalent. Thus there exists unitary $U:\clh_A\rightarrow \clh_{A'}$ such that
									\begin{enumerate}
										\item\label{subeq1} ${A'}_i^*=UA_i^*U^*$,
										\item\label{subeq2} $\gamma'D_{*,A'}=\gamma D_{*,A}U^{*}$
										\item\label{subeq3} $D_{*,\gamma'}P_{D_{*,\gamma'}}\sigma'v-\gamma' \underline{A'}P_{D_{A'}}\sigma'v=D_{*,\gamma}P_{D_{*,\gamma}}\sigma-\gamma \underline{A}P_{D_A}\sigma$,
										\item\label{subeq4} $D_{A'}P_{D_{A'}}\sigma'v=\left(\underset{i=1}{\overset{d}{\oplus}}{U^*}\right)D_AP_{D_A}\sigma$.
									\end{enumerate}
									From Property (\ref{subeq2}) of $U,$ we have $\underline{B}'=D_{*,A'}{\gamma'}^*D_C=\left(\gamma D_{*,A}U^*\right)^*D_C=UD_{*,A}\gamma^*D_C=U\underline{B}$. Hence
										\[\begin{bmatrix}
											I_{\clh_C}&0\\
											0&U
										\end{bmatrix}	\begin{bmatrix}
											\underline{C}&0\\
											\underline{B}&\underline{A}
										\end{bmatrix}=
										\begin{bmatrix}
											\underline{C}&0\\
											\underline{B}'&\underline{A}'
										\end{bmatrix}\begin{bmatrix}
											I_{\underset{i=1}{\overset{d}{\oplus}}\clh_C}&0\\
											0&\underset{i=1}{\overset{d}{\oplus}}U
										\end{bmatrix}.\]
										Hence $\underline{E}$ and $\underline{E}'$ are unitarily equivalent. From the above equality, we also have 
										\begin{equation*}
											D_E=\begin{bmatrix}
												I_{\underset{i=1}{\overset{d}{\oplus}}\clh_C} &0\\
												0&\underset{i=1}{\overset{d}{\oplus}}U^*
											\end{bmatrix} D_{E'}
											\begin{bmatrix}
												I_{\underset{i=1}{\overset{d}{\oplus}}\clh_C} &0\\
												0&\underset{i=1}{\overset{d}{\oplus}}U
											\end{bmatrix}.
										\end{equation*}
								Now, we claim that
										\[v=\begin{bmatrix}
											I_{\underset{i=1}{\overset{d}{\oplus}}\clh_C}&0\\
											0&\underset{i=1}{\overset{d}{\oplus}}U
										\end{bmatrix}|_{\cld_E}.\]
									To prove our claim, suppose $\theta'$ be the transfer function for $W_2$. Then 
									\begin{align*}
										\theta'=&\left( D_{*,\gamma'}P_{D_{*,\gamma'}}-\gamma' \underline{A'}P_{D_{A'}}\right)\sigma'\begin{bmatrix}
											I_{\underset{i=1}{\overset{d}{\oplus}}\clh_C}&0\\
										0&\underset{i=1}{\overset{d}{\oplus}}U
										\end{bmatrix}|_{\mathcal{D}_E}\\
										+&\gamma' D_{*,A'}({I_{\Gamma\otimes \clh_1}}-(\underline{R}\otimes I_{\clh_1})(I_{\Gamma} \otimes A^{'*}))\mathcal{D}_{A'}
										P_{\mathcal{D}_{A'}}\sigma'
										\begin{bmatrix}
									I_{\underset{i=1}{\overset{d}{\oplus}}\clh_C}&0\\
									0&\underset{i=1}{\overset{d}{\oplus}}U
							\end{bmatrix}|_{\mathcal{D}_E}.
									\end{align*}
									With a similar computation, as in the proof of Proposition \ref{prop reduces lifting equivalent}, we get that 
									\begin{align*}
										\theta'\mathcal{D}_E=(\mathcal{D}_{*,\gamma}P_{\mathcal{D}_{*,\gamma}}-\gamma AP_{\mathcal{D}_A})\sigma\mathcal{D}_E=\theta_{C,E}\mathcal{D}_E.
									\end{align*}
									But, we know that $\theta_{C,E}\mathcal{D}_E=\theta_{C,E'}v\mathcal{D}_E$. By our assumption $E'$ is minimal isometric lifting of $C$ and as a result $\theta_{C,E'}$ is injective. Hence 
										\[v=\begin{bmatrix}
										I_{\underset{i=1}{\overset{d}{\oplus}}\clh_C}&0\\
										0&\underset{i=1}{\overset{d}{\oplus}}U
									\end{bmatrix}|_{\cld_E}.\]
									This completes our proof.
								\end{proof}

								\begin{remark}
									Let $\underline{E}$ and $\underline{E}'$ be as defined in Proposition \ref{prop reduces lifting equivalent}. It is clear from the proof of the above proposition, that if there exists a unitary $U:\clh_A\rightarrow \clh_{A'}$ satisfying
									\begin{align}
										\begin{bmatrix}
											I_{\clh_C}&0\\
											0&U
										\end{bmatrix}\underline{E}=\underline{E}'\begin{bmatrix}
											I_{\underset{i=1}{\overset{d}{\oplus}}\clh_C}&0\\
											0&\underset{i=1}{\overset{d}{\oplus}}U
										\end{bmatrix},
									\end{align}
									then \begin{align*}
										\theta_{C,E}=\theta_{C,E'}\begin{bmatrix}
											I_{\underset{i=1}{\overset{d}{\oplus}}\clh_C}&0\\
											0&\left(\underset{i=1}{\overset{d}{\oplus}}U\right)
										\end{bmatrix}|_{\cld_E}.
									\end{align*}
								\end{remark}

								Let $\underline{C}=(C_1,C_2,\ldots, C_d)$ be a row contraction on a Hilbert space $\clh_C$ and $M_{\theta}: \Gamma \otimes \cld \to \Gamma \otimes \cld_C$ be a contractive multi-analytic operator with an injective symbol $\theta$.  Let $\underline{E} = \begin{bmatrix}
									\underline{C} &0\\
									\underline {B} & \underline{A}
								\end{bmatrix}$ be a minimal contractive lifting of $\underline{C}$ on the Hilbert space $\clh_E=\clh_C \oplus \clh_A$ associated to $\theta$ (cf. page 34 of \cite{DEY minimal}). Then by Proposition \ref{Prop lifting} there exists a  contraction $\gamma:\mathcal{D}_{*,A}\rightarrow \mathcal{D}_C$ (cf. Equation (\ref{eqn gamma})) such that $\underline{B}^*=D_C\gamma D_{*,A}.$ Since $\underline{E}$ is a minimal contractive lifting of $\underline{C},$ and we get that $\gamma$ is resolving and $\underline{A}$ is c.n.c. Let the symbol of the characteristic function of $\underline{A}$ be denoted by $\theta_A.$ Since $\underline{A}$ is c.n.c., it was establised in \cite{POP2} that $M_{\theta_A}$ is purely contractive and satisfies the Szeg\"o condition. By Corollary 3.11 of \cite{DEY minimal}, $M_{\theta}$ is equivalent to the characteristic function of the lifting $\underline{E}$ of $\underline{C}.$ By Lemma \ref{characteristic function lemma}, $ \begin{bmatrix}
									D_{*,\gamma}&(I_{\Gamma}\otimes\gamma)\theta_A
								\end{bmatrix}{\sigma}$ is a decomposition of $\theta_{C,E}$ 
								where $\sigma:\mathcal{D}_E\rightarrow\mathcal{D}_{*,\gamma}\oplus\mathcal{D}_A$ is a unitary of the form given by Equation (\ref{eqn for sigma}) and hence
								\begin{equation} \label{decomposition eq}
									\theta \simeq \begin{bmatrix}
										D_{*,\gamma}&(I_{\Gamma}\otimes\gamma)\theta_A
									\end{bmatrix}{\sigma}
								\end{equation} 
								
								Conversely, let $\underline{C}$ be a row contraction on a Hilbert space $\clh_C$ and $\tilde{\theta}: \clm \to \Gamma \otimes \cln$ be the symbol of a purely contractive multi-analytic operator satisfying the Szeg\"{o} condition.
								The model space  is defined as $\clh_{\tilde{\theta}}:=(\Gamma \otimes \cln) \oplus \overline{\Delta_{\tilde{\theta}} (\Gamma \otimes \clm)}$ where $\Delta_{\tilde{\theta}}:=(I-M^*_{\tilde{\theta}} M_{\tilde{\theta}})^{\frac{1}{2}}.$ We define operators $A_i$ for $i=1,2,\dots,d$ on the Hilbert space
								\begin{equation}
									\clh_A= \clh_{\tilde{\theta}} \ominus \{ M_{\tilde{\theta}} \xi \oplus \Delta_{\tilde{\theta}} \xi: \xi \in \Gamma \otimes \cld\}
								\end{equation}
								by
								\begin{equation}
									A_i^*(h \oplus \Delta_{\tilde{\theta}} k)=  (L_i \otimes I)^* h \oplus F^*_i \Delta_{\tilde{\theta}}  k
								\end{equation}
								where the operator $F_i$ is defined by $F_i \Delta_{\tilde{\theta}} l= \Delta_{\tilde{\theta}} (L_i \otimes I) l$ for $l \in  \Gamma \otimes \clm.$ It was shown in Proposition 5.1 of \cite{POP2}, that the tuple $\underline{A}$ is c.n.c.  and there are unitaries $U: \clm \to \cld_A$ and $\tilde{U}: \cln \to \cld_{*,A}$ such that 
								\[ \theta_A= (I \otimes \tilde{U}) \tilde{\theta} U^*.\]
								Let $\gamma: \cld_{*,A} \to \cld_C$ be a resolving contraction. Define $\underline{B}= D_{*,A} \gamma^* D_C$ and
								$\underline{E} = \begin{bmatrix}
									\underline{C} &0\\
									\underline {B} & \underline{A} \end{bmatrix}.$ It follows from Proposition \ref{Prop lifting} and Definition \ref{reduced} that $\underline{E}$ is a minimal contractive lifting. Then by Lemma \ref{characteristic function lemma}
								\[ \theta_{C,E}= \begin{bmatrix}
									D_{*,\gamma}&(I_{\Gamma}\otimes\gamma \tilde{U})\tilde{\theta}
								\end{bmatrix}{\sigma}\]
								for a unitary $\sigma:\mathcal{D}_E \rightarrow \mathcal{D}_{*,\gamma} \oplus U^*\mathcal{D}_A.$ Note that by Proposition 3.12 of \cite{DEY minimal}, $\theta_{C,E}$ is injective.

								\begin{proposition} \label{prop_para}
									Let $\underline{C}$ be a row contraction on the Hilbert space $\clh_C.$ Suppose $M_{\theta}: \Gamma \otimes \cld \to \Gamma \otimes \cld_C$ and $M_{\hat{\theta}}: \Gamma \otimes \clm \to \Gamma \otimes \cld_C$ be two contractive multi-analytic operators with injective symbols.  Let $\underline{E} = \begin{bmatrix}
										\underline{C} &0\\
										\underline {B} & \underline{A} \end{bmatrix}$ and $\underline{\hat{E}} = \begin{bmatrix}
										\underline{C} &0\\
										\underline {\hat{B}} & \underline{\hat{A}} \end{bmatrix}$ be the minimal contractive lifting of $\underline{C}$ on the Hilbert space $\clh_E=\clh_C \oplus \clh_A$ and $\clh_{\hat{E}}=\clh_C \oplus \clh_{\hat{A}}$ respectively, associated to $M_{\theta}$ and $M_{\hat{\theta}}$ respectively, for some Hilbert spaces $\clh_A$ and $\clh_{\hat{A}}$. 
									Then by Proposition \ref{Prop lifting} there exist  contractions $\gamma:\mathcal{D}_{*,A}\rightarrow \mathcal{D}_C$ and $\hat{\gamma}:\mathcal{D}_{*,\hat{A}}\rightarrow \mathcal{D}_C$ such that $\underline{B}^*=D_C\gamma D_{*,A}$ and $\underline{\hat{B}}^*=D_C \hat{\gamma} D_{*,\hat{A}}.$ The symbols $\theta$ and $\hat{\theta}$ are equivalent if and only if there exist unitary operators $U: \cld_A \to \cld_{\hat{A}}$ and $\tilde{U}: \cld_{*,\hat{A}} \to \cld_{*,A}$ such that
									\begin{equation}   \label{parameter}
 \gamma = \hat{\gamma} \tilde{U}^*,\,\, \theta_A U=(I \otimes \tilde{U}) \theta_{\hat{A}}
\end{equation}
								\end{proposition}
								\begin{proof}
									Suppose $\theta$ and $\hat{\theta}$ are equivalent. By Corollary 3.11 of \cite{DEY minimal}, we know that $\theta$ and $\hat{\theta}$ are equivalent to the characteristic functions of liftings $\underline{E}$ and $\underline{\hat{E}},$ respectively.  Therefore, the characteristic functions of liftings $\underline{E}$ and $\underline{\hat{E}}$ are equivalent. By Proposition 3.14 of \cite{DEY minimal} (cf. Theorem 3.7 of \cite{DEY}) it follows that there exist a unitary $u: \clh_{\hat{A}} \to \clh_A$ such that for all $i=1, \ldots, d$
									\[ \begin{bmatrix}
										C_i &0\\
										B_i & A_i \end{bmatrix} 
									\begin{bmatrix}
										I & 0 \\
										0 & u 
									\end{bmatrix}	
									= \begin{bmatrix}
										I & 0 \\
										0 & u 
									\end{bmatrix} 
									\begin{bmatrix}
										C_i &0\\
										\hat{B}_i & \hat{A}_i \end{bmatrix},\]
									i.e., $B_i=u \hat{B}_i$ and $A_iu=u \hat{A}_i.$ Since $\underline{B}= D_{*,A} \gamma^* D_C,$  $\underline{\hat{B}}= D_{*,\hat{A}} \hat{\gamma}^* D_C$ and $D_{*,A} u= u D_{*,\hat{A}},$ we have
									$\gamma = \hat{\gamma} \tilde{U}^*$ where $\tilde{U}=u |_{\cld_{*,\hat{A}}}.$ Also, it follows from \cite{POP2} that $\theta_A U=(I \otimes \tilde{U}) \theta_{\hat{A}}$ where $U=(\bigoplus^d_{i=1} u) |_{\cld_{\hat{A}}}.$
									
									Conversely, suppose there exist unitary operators $U: \cld_A \to \cld_{\hat{A}}$ and $\tilde{U}: \cld_{*,\hat{A}} \to \cld_{*,A}$ such that
									\[ \gamma = \hat{\gamma} \tilde{U}^*,\,\, \theta_A U=(I \otimes \tilde{U}) \theta_{\hat{A}}.\]
									This implies $I -\gamma \gamma ^* = I -\hat{\gamma} \hat{\gamma} ^*$ and hence $D_{*,\gamma}=D_{*,\hat{\gamma}}.$ Moreover, 
									\begin{align*} (I_{\Gamma}\otimes\gamma)\theta_A
										&= (I_{\Gamma}\otimes \hat{\gamma}\tilde{U}^* \tilde{U})\theta_{\hat{A}}U^*\\
										&= (I_{\Gamma}\otimes \hat{\gamma})\theta_{\hat{A}} U^*.
									\end{align*}
									 By Equation (\ref{decomposition eq})
			\[									 \theta \simeq \begin{bmatrix}
										D_{*,\gamma}&(I_{\Gamma}\otimes\gamma)\theta_A
									\end{bmatrix}{\sigma}, \, \,   \hat{\theta} \simeq \begin{bmatrix}
										D_{*,\hat{\gamma}}&(I_{\Gamma}\otimes \hat{\gamma})\theta_{\hat{A}}
									\end{bmatrix}{\hat{\sigma}} 
\]
									where  $\sigma:\mathcal{D} \rightarrow \mathcal{D}_{*,\gamma}\oplus\mathcal{D}_A$ and $\sigma:\mathcal{M} \rightarrow \mathcal{D}_{*,\hat{\gamma}}\oplus\mathcal{D}_{\hat{A}}$ are unitaries. Let $v=\hat{\sigma}^{-1} \begin{bmatrix} I\\ U^* \end{bmatrix} \sigma.$ Then 
$$\begin{bmatrix}
										D_{*,\gamma}&(I_{\Gamma}\otimes\gamma)\theta_A
									\end{bmatrix}{\sigma} = \begin{bmatrix}
										D_{*,\hat{\gamma}}&(I_{\Gamma}\otimes \hat{\gamma})\theta_{\hat{A}}
									\end{bmatrix}{\hat{\sigma}} v.$$ 
Hence $\theta$ and $\hat{\theta}$ are equivalent.
								\end{proof}
								
\begin{remark}
We say that the ordered pairs  $(\gamma, \theta_A)$ and $(\gamma, \theta_{\hat{A}})$ are {\it equivalent} if there exist unitaries $U$ and $ \tilde{U}$ such that the Equation (\ref{parameter}) holds. From Proposition \ref{prop_para}  it follows that for a given row contraction $\underline{C},$ the set of contractive multi-analytic operators $M_{\theta}: \Gamma \otimes \cld \to \Gamma \otimes \cld_C$  with  injective symbols are parametrized by  the ordered pairs $(\gamma, \theta_A)$ corresponding to the liftings of $\underline{C}$ associated to $M_{\theta},$ where $\gamma$ is resolving, and  $\theta_A$ is purely contractive and satisfies the Szeg\"o condition.
\end{remark}
								

								\section{Examples}
								In this section, we discuss a few examples illustrating the results of previous sections.
								\begin{example}\label{distinct}
									Let  $C=\displaystyle\frac{1}{2}$ be the contraction on $\clh_C = \mathbb{C}$.  Then $D_C = \displaystyle \frac{\sqrt{3}}{2}$ and  $\cld_C = \mathbb{C}.$ For $|\alpha| <  1$, consider the Schur function $$\theta(z)= \displaystyle \frac{z-\alpha}{1-\bar{\alpha}z}.$$   
									In section 3 of \cite{DEY minimal}, the notion of ``the lifting associated to a multi-analytic operator" was introduced. Let $$E=\begin{bmatrix}
										C&0\\B&A
									\end{bmatrix}$$ on Hilbert space $\clh_{E}=\clh_C\oplus \clh_{A}$ be a lifting of $C$ associated with $\theta$. Since $\theta$ is inner, we have $\Delta_{\theta}=(I-M_\theta^* M_\theta)^{\frac{1}{2}} =0$. Following the construction described after Corollary $3.4$  of \cite{DEY minimal}, we have $E^*=(V^C)^*|_{\clh_{E}}$ and $\clh_A= \clh^2(\mathbb{D}) \ominus \theta \clh^2(\mathbb{D}),$ where $V^C$ denote the minimal isometric dilation of $C$ and $\clh^2(\mathbb{D})$ is the space of all holomorphic functions on $\mathbb{D}$, whose Fourier coefficients are square summable. The space $\theta \clh^2(\mathbb{D})$ consists of functions vanish at $\alpha,$ i.e.,
									\[\theta \clh^2(\mathbb{D}) = \{ g \in \clh^2(\mathbb{D}): g(\alpha) =0  \}.\] 
									Thus $\clh_{A}= \overline{span} \{ k_{\alpha} \}$	where $k_{\alpha}= \displaystyle \frac{1}{1- \bar{\alpha}z}$.   
									Let $h=(1- |\alpha|^2)^\frac{1}{2}k_{\alpha}\in \clh_A.$ For $\beta, \beta' \in \mathbb{C}, f \in \clh^2(\mathbb{D}),$ we have
									$$\langle (V^C)^*(\beta \oplus h), \beta' \oplus f \rangle=\langle (\beta + h,\frac{1}{2} \beta' + \frac{\sqrt{3}}{2} \beta' \oplus z f \rangle.$$	
									Hence
									$B=\displaystyle \frac{\sqrt{3}}{2}(1- |\alpha|^2)^\frac{1}{2}$ and 
									$A= \alpha.$ Consequently, $D_{*,A} =D_A = (1-|\alpha|^2)^{\frac{1}{2}}, D_C = \displaystyle \frac{\sqrt{3}}{2}. $ Furthermore,
									\[\gamma = 1,\,\, D_{*,\gamma} = 0,\]
									and
									\[
									D_{E} =(1-\frac{1}{4}|\alpha|^2)^{-\frac{1}{2}} 
									\begin{bmatrix}
										\frac{3}{4}|\alpha|^2& -\frac{\sqrt{3}}{2}(1-|\alpha|^2)^{\frac{1}{2}}\alpha\\
										-\frac{\sqrt{3}}{2}(1-|\alpha|^2)^{\frac{1}{2}}\bar{\alpha}	& 1-|\alpha|^2
									\end{bmatrix}.\]
									From  Equation (\ref{eqn for sigma}) we obtain the map $\sigma: \cld_E \to \cld_A$ as
									\[ \sigma D_{E} = \begin{bmatrix}
										0&0\\
										-\bar{\alpha}\frac{\sqrt{3}}{2}& (1-|\alpha|^2)^{\frac{1}{2}}
									\end{bmatrix}.
									\]
									Since both $\cld_E$ and $\cld_A$ are one dimensional and satisfy $\norm{\sigma D_Eh} = \norm{D_E h},$ the map $\sigma$ is unitary from $\cld_E$ to $\cld_A$. 
									
									It is easy to check that the characteristic function $\theta_A$ of $A$ is equal to $\theta.$ The symbol $\theta_{C,E}$ of the characteristic function of the lifting $E$ of $C$ can be written as
									\[\theta_{C,E}= \begin{bmatrix}
										0 & \theta_A
									\end{bmatrix} \sigma= \displaystyle \frac{(z-\alpha)\sigma}{1-\bar{\alpha}z,}\]
									which coincides with the $\theta$. 
									
									The colligation matrix corresponding to the lifting $E$ is
									\begin{align*}
										V&=\begin{bmatrix}
											A^*& D_AP_{D_A}\sigma\\
											\gamma D_{*,A}& (D_{*,\gamma}P_{D_{*,\gamma}}\sigma-\gamma A P_{D_A}\sigma)
										\end{bmatrix}\\
										&=\begin{bmatrix}
											\bar{\alpha}& (1-|\alpha|^2)^{\frac{1}{2}}\sigma\\
											(1-|\alpha|^2)^{\frac{1}{2}}&-\alpha\sigma
										\end{bmatrix}.
									\end{align*}
									Using Equation (\ref{transfer}), the transfer function of $V$ is computed as 
									\begin{align*}
										\theta_V &= -\alpha\sigma+ (1-|\alpha|^2)^{\frac{1}{2}}(1- \bar{\alpha} z)^{-1}(z(1-|\alpha|^2)^{\frac{1}{2}}\sigma)\\
										&=-\alpha\sigma+\frac{z(1-|\alpha|^2)\sigma}{1-\bar{\alpha}z}
										=\frac{(z-\alpha)\sigma}{1-\bar{\alpha}z}.
									\end{align*}
								\end{example}
								\begin{example}\label{example2}
									Let $C = \displaystyle \frac{1}{2},\, A  = \displaystyle \frac{1}{2}$ are the operators on $\clh_C=\clh_A=\mathbb{C}$, respectively. Consider the lifting $E$ of $C$ defined by
									\[E = \begin{bmatrix}
										C&0\\B&A
									\end{bmatrix} = \frac{1}{2}\begin{bmatrix}
										1 & 0\\
										1 & 1
									\end{bmatrix}\] on $\clh_E=\mathbb{C}^2$. Then  $D_C = \displaystyle \frac{\sqrt{3}}{2}$, $D_{*, A} = \displaystyle \frac{\sqrt{3}}{2}$,  and 
									\[
									{D_E} 
									= \frac{1}{2(5+2 \sqrt{5}))^{\frac{1}{2}}}\begin{bmatrix}
										\sqrt{5}+2 & -1\\
										-1 & \sqrt{5}+3
									\end{bmatrix}.\]
									\text{From  Equation (\ref{eqn for sigma})},\begin{align*}
										\sigma D_E &= \begin{bmatrix}
											\frac{\sqrt{5}}{2\sqrt{3}}&0\\
											-\frac{1}{2\sqrt{3}}&\frac{\sqrt{3}}{2}
										\end{bmatrix},\;\text{solving for}\; \sigma\;\text{we obtain}\\
										\sigma&=\frac{1}{\sqrt{3}(5+2 \sqrt{5}))^{\frac{1}{2}}}\begin{bmatrix}
											\sqrt{5}+3&1\\
											-1&\sqrt{5}+3
										\end{bmatrix}
									\end{align*}	
									Using $B^* = {D_C}\gamma D_{*, A}$ we get that, $\gamma = \displaystyle\frac{2}{3}$ and $D_{*,\gamma} = \displaystyle \frac{\sqrt{5}}{3}$. For $h_C\in \clh_C$ the characteristic function $\theta_{C,E}$ is given by
									\begin{align*}
										\theta_{C,E}(z)(D_E)h_C 
										&= D_Ch_C- \gamma D_{*,A}Bh_C-\underset{n = 1}{\overset{\infty}{\sum}} \gamma D_{*,A}(A^*)^nBh_C z^n\\
										&= \frac{\sqrt{3}}{2}h_C- \underset{n = 0}{\overset{\infty}{\sum}}z^n\frac{1}{2\sqrt{3}}  \left(\frac{1}{2} \right)^{n}h_C\\
										&=\frac{\sqrt{3}}{2}h_C- \frac{1}{2\sqrt{3}(1-\frac{z}{2})}h_C\\
										&=\frac{4-3z}{4\sqrt{3}(1-\frac{z}{2})}h_C,
									\end{align*}
									and for $h_A\in \clh_A$, we have
									\begin{align*}
										\theta_{C,E}(z)(D_E)h_A &=-\gamma AD_Ah_A+z\underset{n = 0}{\overset{\infty}{\sum}} \gamma D_{*,A}(A^*)^nP_jD_A^2h_A z^n\\
										&=- \frac{1}{2\sqrt{3}}h_A + z\underset{n = 0}{\overset{\infty}{\sum}} z^n  \left(\frac{\sqrt{3}}{4} \right) \left(\frac{1}{2} \right)^{n} h_A\\
										&=-\frac{1}{2\sqrt{3}}h_A+\frac{\sqrt{3}z}{4(1-\frac{z}{2})}h_A\\
										&= \frac{(z-\frac{1}{2})}{\sqrt{3}(1-\frac{z}{2})}h_A.
									\end{align*}
									Since $A=\displaystyle \frac{1}{2}$, the characteristic function $\theta_A(z)$ is equal to $\displaystyle \frac{z-\frac{1}{2}}{1-\frac{z}{2}}$ and
									\[
									\begin{bmatrix}
										D_{*,\gamma} & \gamma\theta_A
									\end{bmatrix}\sigma D_E= \begin{bmatrix}
										\frac{\sqrt{5}}{3} & \frac{2(z-\frac{1}{2})}{{3(1-\frac{z}{2})}}
									\end{bmatrix}\begin{bmatrix}
										\frac{\sqrt{5}}{2\sqrt{3}}&0\\
										-\frac{1}{2\sqrt{3}}&\frac{\sqrt{3}}{2}
									\end{bmatrix}=\begin{bmatrix}
										\frac{4-3z}{4\sqrt{3}(1-\frac{z}{2})}&
										\frac{(z-\frac{1}{2})}{\sqrt{3}(1-\frac{z}{2})}
									\end{bmatrix}=\theta_{C,E}(z) D_E .
									\]	
									Also,
									\[
									\theta_{C,E}(z) = 	\begin{bmatrix}
										D_{*,\gamma} & \gamma\theta_A
									\end{bmatrix}\sigma = \begin{bmatrix}
										\frac{2(\sqrt{5}+2) - z(\sqrt{5}+3)}{2(1-\frac{z}{2})}&\frac{(\sqrt{5}+4)z-2}{2(1-\frac{z}{2})}
									\end{bmatrix}.
									\]
									The associated colligation $V:\clh_A\oplus\cld_E\rightarrow \clh_A\oplus\cld_C$ is
									\[
									V= \begin{bmatrix}
										A^*&D_AP_{D_A}\sigma\\
										\gamma D_{*,A}& D_{*,\gamma}P_{D_{*,\gamma}}\sigma -\gamma AP_{D_A}\sigma
									\end{bmatrix}=\begin{bmatrix}
										A^*&0&D_A\\
										\gamma D_{*,A}& D_{*,\gamma}&-\gamma A
									\end{bmatrix}\begin{bmatrix}
										I&0\\
										0&\sigma 
									\end{bmatrix}\]
									After simplification, we get
									\[								V=\begin{bmatrix}
										\frac{1}{2}&-\frac{1}{2(5+2 \sqrt{5})^{\frac{1}{2}}}&\frac{\sqrt{5}+3}{2(5+2 \sqrt{5})^{\frac{1}{2}}}\\
										\frac{1}{\sqrt{3}}&\frac{(\sqrt{5}+2)}{\sqrt{3}(5+2 \sqrt{5})^{\frac{1}{2}}}&-\frac{1}{\sqrt{3}(5+2 \sqrt{5})^{\frac{1}{2}}}
									\end{bmatrix}
									\]
									Similar to Example(\ref{example2}) comparing with standard form of colligation we obtain,
									$A=\displaystyle \frac{1}{2}$, $B = \begin{bmatrix}
										-\frac{1}{2(5+2 \sqrt{5})^{\frac{1}{2}}}&\frac{\sqrt{5}+3}{2(5+2 \sqrt{5})^{\frac{1}{2}}}
									\end{bmatrix}$, $C = \frac{1}{\sqrt{3}}$, $D = \begin{bmatrix}
										\frac{(\sqrt{5}+2)}{\sqrt{3}(5+2 \sqrt{5})^{\frac{1}{2}}}&-\frac{1}{\sqrt{3}(5+2 \sqrt{5})^{\frac{1}{2}}}
									\end{bmatrix}.$ Now, \begin{align*}
										\theta_V(z)  = \begin{bmatrix}
											\frac{(\sqrt{5}+2)}{\sqrt{3}(5+2 \sqrt{5})^{\frac{1}{2}}}&-\frac{1}{\sqrt{3}(5+2 \sqrt{5})^{\frac{1}{2}}}
										\end{bmatrix}+\frac{1}{\sqrt{3}}(1-\frac{z}{2})^{-1}z\begin{bmatrix}
											-\frac{1}{2(5+2 \sqrt{5})^{\frac{1}{2}}}&\frac{\sqrt{5}+3}{2(5+2 \sqrt{5})^{\frac{1}{2}}}
										\end{bmatrix}.
									\end{align*}
									After simplification, we get
									\[
									\theta_V(z) = \begin{bmatrix}
										\frac{2(\sqrt{5}+2) - z(\sqrt{5}+3)}{2(1-\frac{z}{2})}&\frac{(\sqrt{5}+4)z-2}{2(1-\frac{z}{2})}
									\end{bmatrix} = \theta_{C,E}(z).
									\]
								\end{example}
								\begin{example}
									Let $\theta(z) = \displaystyle \frac{z}{2}$ and $C=\displaystyle \frac{1}{2}.$ So $\Delta_{\theta} = \displaystyle \frac{\sqrt{3}}{2}$. Define the isometry $W:\clh^2(\mathbb{D})\rightarrow \clh^2(\mathbb{D})\oplus\clh^2(\mathbb{D})$ by $$Wf= M_{\theta}f\oplus\Delta_{\theta}f.$$  Define the subspace $\clh_{A} = (\clh^2(\mathbb{D})\oplus \clh^2(\mathbb{D}))\ominus W\clh^2(\mathbb{D}) = \{g\oplus-\frac{M^*_zg}{\sqrt{3}}:g\in\clh^2(\mathbb{D})\}.$ By setting $$W|_{\clh_C} =I_{\clh_C}$$ we can extend $W$ as a unitary operator $W:\clh_C\oplus \clh_A\oplus\clh^2(\mathbb{D})\rightarrow \clh_C\oplus\clh^2(\mathbb{D})\oplus\clh^2(\mathbb{D}).$   Clearly $\overline{\Delta_{\theta}\clh^2(\mathbb{D})} = \clh^2(\mathbb{D}).$ Define $Y=M_z$ on $\clh^2(\mathbb{D}).$ Then $Y\Delta_{\theta}f=\Delta_{\theta}M_zf.$  Now define 
$$\tilde{V} := V^C\oplus Y:\hat{\clh}_C\oplus\clh^2(\mathbb{D})\rightarrow\hat{\clh}_C\oplus\clh^2(\mathbb{D})$$
where $V^C$  is the minimal isometric dilation of $C$ on the Hilbert space $\hat{\clh}_C$ (cf. Section 6 of \cite{DEY})  
Let $E^*=\tilde{V}^*|_{\clh_C\oplus \clh_A}.$ We observe that $E$ is a lifting of $C$.	Let $\alpha \in \clh_C$ and $g(z)=\sum_{n=0}^\infty g_n z^n \in \mathcal{H}^2.$ 
									Define	$B^*(g\oplus(-\frac{M^*_zg}{\sqrt{3}})) = \frac{\sqrt{3}g_0}{2}$ and $A^*(g\oplus(-\frac{M^*_zg}{\sqrt{3}}))=M^*_zg\oplus (-\frac{M^*_z(M^*_zg)}{\sqrt{3}}).$ Then
									\begin{align*}
										B\alpha &= \frac{\sqrt{3}}{2}\alpha\oplus 0,\\
										A(g\oplus(-\frac{M^*_zg}{\sqrt{3}})) &= g_A\oplus(-\frac{M^*_z g_{A}}{\sqrt{3}}),
									\end{align*}
									where $g_A = \frac{3}{4}g_0z+\sum_{n=2}^{\infty}g_{n-1}z^n.$ It follows that
									\begin{align*}
										D_{*,A}(g\oplus(-\frac{M^*_zg}{\sqrt{3}}))&=g_0+\frac{g_1}{2}z\oplus(-\frac{M^*_z(g_0+\frac{g_1}{2}z)}{\sqrt{3}})\\
										D_A(g\oplus(-\frac{M^*_zg}{\sqrt{3}}))&=\frac{g_0}{2}\oplus 0.
									\end{align*}
									It is easy to check that	$B=D_{*,A}\gamma^*D_C$ where $\gamma:\cld_{*,A}\rightarrow\cld_C$ is a co-isometric contraction defined by $\gamma D_{*,A}(g\oplus(-\frac{M^*_zg}{\sqrt{3}})) = g_0.$ Hence
									\[
									E=\begin{bmatrix}
										C&0\\
										B&A
									\end{bmatrix}
									\]	is a lifting of $C$. For $\alpha\in \clh_C$, $D_E\alpha = 0$, and 
									\[
									\cld_E = \overline{span}\{D_E(g\oplus(-\frac{M^*_zg}{\sqrt{3}})):g\oplus(-\frac{M^*_zg}{\sqrt{3}})\in \clh_A\}.
									\]
									Since $D_E(g\oplus(-\frac{M^*_zg}{\sqrt{3}})) = D_A(g\oplus(-\frac{M^*_zg}{\sqrt{3}})) = \frac{g_0}{2},$ we obtain 
									\begin{align*}
										\sigma D_E = \begin{bmatrix}
											D_{*,\gamma}D_C&0\\
											-A^*\gamma^*D_C&D_A
										\end{bmatrix}=\begin{bmatrix}
											0&0\\-A^*\gamma^*D_C&D_A
										\end{bmatrix}.
									\end{align*}
									Thus,  $\sigma D_E(g\oplus(-\frac{M^*_zg}{\sqrt{3}})) = D_A(g\oplus(-\frac{M^*_zg}{\sqrt{3}})) = \frac{g_0}{2}$. Therefore, $\sigma = I$ is the unitary operator from $\cld_E$ to $\cld_A$. For $\alpha \in \clh_C$
									\begin{align*}
										\theta_{C,E}(z)D_E\alpha=D_C\alpha-\gamma D_{*,A}B\alpha -\underset{n = 1}{\overset{\infty}{\sum}} D_{*,A}{A^*}^nB\alpha z^n = 0,\end{align*}
									and for $h_a=g\oplus(-\frac{M^*_zg}{\sqrt{3}})\in \clh_A$
									\begin{align*}
										\theta_{C,E}(z)D_Eh_a=-\gamma AD_Ah_a+ \underset{n = 1}{\overset{\infty}{\sum}}\gamma D_{*,A}A_{n}^*P_jD_A^2h_a z^n=\frac{g_0}{4}z.
									\end{align*}
									So,
									\begin{align*}
										\theta_{C,E}(z) &= \begin{bmatrix}
											0&\frac{g_0}{4}z
										\end{bmatrix}\\
										&=\begin{bmatrix}
											D_{*,\gamma}&\gamma\theta_A(z)
										\end{bmatrix}\sigma
									\end{align*}
									The associated colligation is
									\begin{align*}
										V&= \begin{bmatrix}
											A^*&D_AP_{D_A}\sigma\\
											\gamma D_{*,A}& D_{*,\gamma}P_{D_{*,\gamma}}\sigma-\gamma AP_{D_A}\sigma
										\end{bmatrix}\\
										&= \begin{bmatrix}
											A^*&D_A\sigma\\
											\gamma D_{*,A}& \gamma A\sigma
										\end{bmatrix}.
									\end{align*} For $h(z)=\sum_{n=0}^\infty h_n z^n$ and $\tilde{h}=h \oplus (-\frac{M^*_z h}{\sqrt{3}})$
									\begin{align*}
										V\begin{pmatrix}
											g\oplus(-\frac{M^*_zg}{\sqrt{3}})\\
											D_A(\tilde{h})
										\end{pmatrix}  &= \begin{bmatrix}
											A^*&D_A\sigma\\
											\gamma D_{*,A}& \gamma A\sigma
										\end{bmatrix}\begin{pmatrix}
											g\oplus(-\frac{M^*_zg}{\sqrt{3}})\\
											D_A(\tilde{h})
										\end{pmatrix} \\&=\begin{pmatrix}
											M^*_zg\oplus (-\frac{M^*_z(M^*_zg)}{\sqrt{3}})+\frac{h_0}{4}\oplus 0\\
											g_0\end{pmatrix}.
									\end{align*}
								\end{example}

								\section{Colligations with certain class of basic operators}
								
								In this section, we describe the structure of a co-isometric observable colligation  having finite dimensional input space where the dimension of the input space is equal to the dimension of the defect space of the basic operator of the colligation. We begin with a simple application of linear algebra which we would need in the next theorem.
								\begin{lemma}\label{defect dimension lemma}
									Let $\underline{T} = (T_1,T_2,\ldots,T_d)$ be a row contraction on a Hilbert space $\clh$ with $dim(\clh)=n<\infty$. Then
									\[
									n(d-1)\leq dim(\cld_T)\leq nd.
									\]
									Moreover, on a Hilbert space $\clh$ of dimension n, for any integer $k$ such that $n(d-1) \leq k \leq nd,$ there exists a row contraction $\underline{G}=(G_1, \ldots, G_d)$ on $\clh$ with $dim(\cld_{G}) = k.$   
								\end{lemma}
									\begin{proof}
										Let $\underline{T} = (T_1,T_2,\ldots,T_d)$ be a row contraction defined on $\clh$ of dimension $n.$ That is, $\underline{T}=[T_1,T_2,\ldots,T_d]$ is a row matrix with domain $\overset{d}{\underset{1}{\oplus}}\clh$ and range $\clh.$ So, $ran(\underline{T})\leq n.$ By rank-nullity theorem, $ker(\underline{T})\geq nd-n.$
										Since \[
										ran(\underline{T}^*\underline{T}) = ran(\underline{T}^*)=ker(\underline{T})^{\perp},
										\]
										$ran(\underline{T}^*\underline{T})\leq nd-(nd-n)=n.$ Thus, $$nd \geq  ran(I-\underline{T}^*\underline{T})\geq nd-n = n(d-1).$$
Since $I-\underline{T}^*\underline{T}$ is a positive matrix, it follows immediately from spectral theorem that $nd \geq  dim(\cld_T) \geq nd-n = n(d-1).$ This proves the first assertion.
Next, we assume $k\in\mathbb{N}$ satisfies $$n(d-1)\leq k\leq nd.$$ Let $r=nd-k.$ Then $0\leq r\leq n.$ Let $\clh$ be a Hilbert space with dimension $n.$ Define the row contraction $\underline{G}=(G_1, \ldots, G_d)$ on $\clh$ with $$G_1  = \begin{bmatrix}
											I_{r\times r}&0\\
											0&0
										\end{bmatrix}\:\: \text{and}\: G_j=0\: \text{for}\: 1\leq j\leq d.$$  Then $ran(G^*G) = r$ and hence $dim(\cld_G)=ran(I-G^*G) = nd-r = k.$  The lemma follows.
								\end{proof}
								In the following theorem we characterize (cf. \cite[Theorem 2.2]{Arlinski}) certain colligation matrices  using the theory of characteristic function of lifting developed in Section  \ref{lrc} and in \cite{DEY minimal}:
								\begin{theorem}\label{structure of colligation}
									Let $\underline{T} = (T_1,T_2,\ldots,T_d)$ be a row contraction on a Hilbert space $\clh_{T}$ with finite dimensional defect space $\cld_T$. Let $$W=\begin{bmatrix}
										\underline{T}^*&F\\G&H
									\end{bmatrix}=\begin{bmatrix}
										T^*_1&F_1\\
										T^*_2&F_2\\
										\vdots&\vdots\\
										T^*_d&F_d\\
										G&H
									\end{bmatrix}:\clh_{T}\oplus \clu\rightarrow \clh^d_{T}\oplus \cly$$ be a co-isometric observable colligation with basic operator $\underline{T}^*$.  If the input space $\clu$ be such that $dim(\clu) = dim(\cld_T)$, then $W$ takes  the form 
									\[
									W = \begin{bmatrix}
										I&0\\0&\gamma'
									\end{bmatrix}\begin{bmatrix}
										\underline{T}^*&D_T\tilde{\sigma}\\D_{*,T}&-\underline{T} \tilde{\sigma}
									\end{bmatrix},
									\]
									where $\gamma'$ is a co-isometric contraction from $\cld_{*,T}$ to $\cly$  and $\tilde{\sigma}$ is a unitary operator from $\clu$ to $\cld_{T}$.
								\end{theorem}
								\begin{proof}
									Since $W$ is co-isometric colligation, we have $WW^*=I$. This implies
									$$\begin{bmatrix}
										T^*_1&F_1\\
										T^*_2&F_2\\
										\vdots&\vdots\\
										T^*_d&F_d\\G&H
									\end{bmatrix}
									\begin{bmatrix}
										T_1&T_2&\ldots& T_d&G^*\\
										F^*_1&F^*_2&\ldots&F^*_d&H^*
									\end{bmatrix} = I.
									$$  From this we get
									$\norm{F^*h}^2 = \norm{D_T h}^2$ for $h=(h_1,h_2,\ldots,h_d)\in \clh_{T}^d$. Then $	F^* = KD_T$   for some isometry $K:\cld_{T}\rightarrow\clu.$ 
									Since $\dim(\cld_{T}) = \dim(\clu)$ and are finite, it follows that $K$ is a unitary operator from $\cld_{T}$ to $\clu$. Since $D_T|_{\cld_{T}}$ is injective,  $F=D_TK^*$ is also injective.	
									\par Now, we claim that the transfer function $\Theta=H+G(I_{\Gamma\otimes \clh_T}-R\otimes \underline{T}^*)^{-1}(R\otimes F)$ is injective.
									Let  $\Theta(u)= [H+G(I-R\otimes \underline{T}^*)^{-1}(R\otimes F)](u)=0$ for some $u\in\clu.$ Then,
									\begin{align*}
										H(u) = 0\;& \text{ and }\; G(I-R\otimes \underline{T}^*)^{-1}(R\otimes F)(u)=0.
									\end{align*}
									The colligation $W$ is observable and $F:\clu\rightarrow \clh^d_T$ is injective implies  that  the operator $G(I-R\otimes \underline{T}^*)^{-1}(R\otimes F)$ is injective on $\clu$. Thus,  $u=0$ and hence the claim follows.
									
									If $\cly$ is finite dimensional and satisfy $n(d-1)\leq dim{\cly}\leq nd$  for some $n\in\mathbb{N},$ by Lemma \ref{defect dimension lemma} there exists a row contraction $\underline{C} = (C_1,C_2,\ldots,C_d)$  on $\cly$ with $\dim(\cld_{C}) = \dim(\cly)$. In particular, we can define row contraction $\underline{C}$ such that $\cld_{C} = \cly.$
									If $\cly$ is infinite dimensional space then define  $\underline{C} = (\frac{1}{2}I,0,\ldots,0)$ on $\cly.$ In both cases, we obtain a row contraction $\underline{C}$ with $\cld_{C} = \cly.$

									Then, $\Theta$ is multi-analytic operator with injective symbol from $\Gamma\otimes\clu$ to $\Gamma\otimes\cld_{C}.$
									From the construction given in \cite{DEY minimal},  it follows that there exist a minimal  contractive lifting of $\underline{C}$ $$\underline{E} = \begin{bmatrix}
										\underline{C}&0\\
										\underline{B}&\underline{A}
									\end{bmatrix}$$ on $\clh_{E} = \clh_C\oplus\clh_{A}$ with $\dim(\cld_{E}) = \dim(\clu)$ and the  characteristic function $M_{C,E}=\Theta$.
									
									By Lemma \ref{lemma characteristic fun of colligation}, corresponding to the minimal contractive lifting $\underline{E},$ there exist  a co-isometric observable colligation $V,$ 
									\[
									V= \begin{bmatrix}
										\underline{A}^*& D_{A}P_{D_A}\sigma\\
										\gamma D_{*,A}&(D_{*,\gamma}P_{D_{*,\gamma}}\sigma-\gamma \underline{A} P_{D_A}\sigma)
									\end{bmatrix}
									\] 
									such that the transfer function $\Theta_V= M_{C,E} = \Theta.$
									
									\par Observe that  $V$ and $W$ are two co-isometric observable colligations with same transfer function. Thus  there exist a unitary $U:\clh_{T}\rightarrow \clh_{A}$ satisfying,
									\begin{equation}
										UT^*_i = A^*_iU \text{ for }i = 1,2,\ldots,d.
									\end{equation}
									From this we obain the following relations,
									\begin{align}\label{eqn 5.1}
										D_T&= (\underset{i=1}{\overset{d}{\oplus}}U^*) D_{A} (\underset{i=1}{\overset{d}{\oplus}}U)\\\label{equivalence of D_*T and D_*A}  D_{*,T}& = U^*D_{*,A}U.
									\end{align}
									This implies $\dim(\cld_{A}) = \dim(\cld_{T})=\dim(\clu)$.  Since $M_{C,E} = \Theta$, 	 we get $\cld_{A} = \cld_{E}$. Given that $\sigma$ is a unitary operator from $\cld_{E}$ to $\cld_{A}\oplus\cld_{*,\gamma}$, we conclude that
									$\cld_{*,\gamma} = \{0\}$, so $\gamma$ is a co-isometry, and
									\[
									V=\begin{bmatrix}
										\underline{A}^*&D_{A}\sigma\\
										\gamma D_{*,A}&-\gamma\underline{A}\sigma
									\end{bmatrix}.
									\]
									By unitarily equivalence between $V$ and $W$, and by Equations (\ref{eqn 5.1}) and (\ref{equivalence of D_*T and D_*A}), we have
									\begin{enumerate}
										\item $\left(\underset{i=1}{\overset{d}{\oplus}}U\right) F = D_{A}\sigma$, equivalently $F=\left(\underset{i=1}{\overset{d}{\oplus}}U^*\right)D_{A}\sigma=D_T\left(\underset{i=1}{\overset{d}{\oplus}}U^*\right)\sigma,$	
										\item $G= \gamma D_{*,A}U$, equivalently $G= \gamma UD_{*,T},$
										\item $ H=\gamma\underline{A}\sigma=\gamma U\underline{T}\left(\underset{i=1}{\overset{d}{\oplus}}U^*\right)\sigma.$
										
									\end{enumerate}
									Therefore
									\[
									W=\begin{bmatrix}
										\underline{T}^*&F\\G&H
									\end{bmatrix} = \begin{bmatrix}
										\underline{T}^*&D_T\left(\underset{i=1}{\overset{d}{\oplus}}U^*\right)\sigma\\\gamma UD_{*,T}&-\gamma U \underline{T} \left(\underset{i=1}{\overset{d}{\oplus}}U^*\right)\sigma
									\end{bmatrix}.
									\]
									Define $\gamma' = \gamma U$ and $\tilde{\sigma} = \left(\underset{i=1}{\overset{d}{\oplus}}U^*\right)\sigma$. Since $U$ is unitary and $\gamma$ is a co-isometry, it follows that $\gamma'$ is a co-isometry from $\cld_{*,T}$ to $\cly$ and $\tilde{\sigma}$ is a unitary operator from $\clu$ to $\cld_{T}$. Hence we conclude that
									\[
									W=  \begin{bmatrix}
										\underline{T}^*&D_T\left(\underset{i=1}{\overset{d}{\oplus}}U^*\right)\sigma\\\gamma' D_{*,T}&-\gamma' \underline{T} \left(\underset{i=1}{\overset{d}{\oplus}}U^*\right)\sigma
									\end{bmatrix}=\begin{bmatrix}
										I&0\\0&\gamma'\end{bmatrix}\begin{bmatrix}
										\underline{T}^*&D_T\tilde{\sigma}\\D_{*,T}&- \underline{T} \tilde{\sigma}
									\end{bmatrix}.
									\]
									This completes the proof.
								\end{proof} 
Next, we show that the representation of the colligation operator obtained in Theorem \ref{structure of colligation} is unique upto certain unitaries and partial isometries.
								\begin{corollary}  Let $W=\begin{bmatrix}
										\underline{T}^*&F\\G&H
									\end{bmatrix}: \clh_{T}\oplus \clu\rightarrow \clh^d_{T}\oplus \cly$ be a co-isometric observable colligation. Assume the $dim(\cld_{T})=dim(\clu)$ is finite.  Then by Theorem \ref{structure of colligation}, $$W=\begin{bmatrix}
										\underline{T}^*&D_T\sigma\\
										\gamma D_{*,T}&-\gamma\underline{T}\sigma
									\end{bmatrix}$$  for some unitary $\sigma\in\clb(\clu,\cld_{T})$   and for some co-isometric contraction $\gamma\in\clb(\cld_{*,T},\cly).$ 
									Let $\tilde{W}$ be another co-isometric observable  colligation with the same basic operator $\underline{T}^*$, input space $\clu$ and output space $\cly$. Then $\tilde{W}$ has the following representation:
									\[ 
									\tilde{W} = \begin{bmatrix}
										\underline{T}^*& D_T \sigma S\\\gamma\nu D_{*,T}&-\gamma\nu \underline{T}\sigma S
									\end{bmatrix}
									\] 
									where $S$ is a unitary operator on $\clu$, $\nu$ is a partial isometry on $\cld_{*,T}$ and $\gamma$ is a co-isometric contraction from $\cld_{*,T}$ to $\cly$.
								\end{corollary}
								\begin{proof} 
									Let $\tilde{W}=\begin{bmatrix}
										\underline{T}^*&\tilde{F}\\\tilde{G}&\tilde{H}
									\end{bmatrix}$ for some contractions $\tilde{F}\in \mathcal{B}(\clu,\clh_T)$, $\tilde{G}\in \mathcal{B}(\clh_T,\cly)$ and $\tilde{H}\in \mathcal{B}(\clu,\cly)$ . By Theorem \ref{structure of colligation}, we know that
									\begin{align*}
										W&=\begin{bmatrix}
											\underline{T}^*&D_T\sigma\\
											\gamma D_{*,T}&-\gamma\underline{T}\sigma
										\end{bmatrix}\\
										\tilde{W}&=\begin{bmatrix}
											\underline{T}^*&D_T\tilde{\sigma}\\
											\tilde{\gamma}D_{*,T}&-\tilde{\gamma}\underline{T}\tilde{\sigma}
										\end{bmatrix}
									\end{align*}
									where  $\gamma, \tilde{\gamma}$  are co-isometric contractions from $\cld_{*,T}$ to $\cly$ and  $\sigma, \tilde{\sigma}$ are unitary operators from $\clu$ to $\cld_{T}.$
									Define $S:= \sigma^*\tilde{\sigma}$,  then $S$ is a unitary on $\clu$ and $\tilde{F} = D_T\tilde{\sigma} =  D_{T}\sigma S $. Since $\gamma$ and $\tilde{\gamma}$ are two co-isometric operators from $\cld_{*,T}$ to $\cly$, there exist an isometry $\mu$ from Range$(\gamma^*)$ to Range$(\tilde{\gamma}^*)$ such that
									$\tilde{\gamma}^* = \mu\gamma^*$. Define $$\nu D_{*,T}h := \begin{cases}
										\mu D_{*,T}h&\text{if} \;D_{*,T}h\in Range(\gamma^*)\\
										0& \text{otherwise} 
									\end{cases}.$$ Then $\nu$ extends to a partial isometry on $\cld_{*,T}$ and $\tilde{\gamma} = \gamma\nu^*.$ Thus,
									\[
									\tilde{W} =\begin{bmatrix}
										\underline{T}^*&D_T\sigma S\\
										\gamma\nu^* D_{*,T}&-\gamma\nu^*\underline{T}\sigma S
									\end{bmatrix},
									\]
									as required.
								\end{proof}

								\section*{Acknowledgment}	The first author is supported by Faculty Research Scheme, IIT (ISM), Dhanbad, grant number FRS(218)/2024-2025/M$\&$C. The second author was supported by SERB MATRICS Grant number MTR/2018/000343. The research of the third author is supported in part by grants from the Indian Institute of Technology Goa (SEED Grant 2022/SG/KH/047) and the Anusandhan National Research Foundation (MATRICS Grant MTR/2022/000339).

								\vspace{1cm}
								
								\noindent Department of Mathematics and Computing, Indian Institute of Technology (ISM), Dhanbad, Jharkhand-826004, India\\
								Email address:  neerusingh41@gmail.com, neerubala@iitism.ac.in

								\vspace{.5cm}
								\noindent Department of Mathematics, Indian Institute of Technology Bombay, Powai, Mumbai-400076, India\\
								Email address: santanudey@iitb.ac.in

								\vspace{.5cm}
								\noindent School of Mathematics \& Computer Science,
Indian Institute of Technology Goa,  \\ 
At Goa College of Engineering Campus, Farmagudi, Ponda, Goa-403401, India\\
								Email address: kalpesh@iitgoa.ac.in

								\vspace{.5cm}
								
								\noindent Institut de Mathématiques de Toulouse,	Université de Toulouse,
								118, route de Narbonne, F-31062 Toulouse Cedex 9, France
								\\
								Email address:  rmazhava@math.univ-toulouse.fr

							\end{document}